\documentclass[a4paper,10pt]{article}

\usepackage[margin=2.5cm]{geometry}
\usepackage{cancel}

\usepackage[utf8]{inputenc}

\usepackage[T1]{fontenc}

\usepackage{xfrac}
\usepackage{braket}
\usepackage[shortlabels]{enumitem}

\usepackage{amsmath}
\usepackage{amsthm}
\usepackage{amstext}
\usepackage{amssymb}
\usepackage[colorlinks, citecolor=blue]{hyperref}
\usepackage{cleveref}
\crefformat{footnote}{#2\footnotemark[#1]#3}
\usepackage{tikz}
\usetikzlibrary{cd,matrix,arrows,mindmap,backgrounds}

\usepackage{mathtools}
\usepackage{stmaryrd}
\usepackage{mathrsfs}
\usepackage{bbold}
\usepackage{extarrows}
\usepackage{graphicx, import}
\usepackage{fancyhdr}
\usepackage{marginnote}
\usepackage{xifthen}
\usepackage{enumitem}
\usepackage{verbatim}
\usepackage{changepage}
\usepackage{csquotes}

\usepackage{marginnote}

\usepackage{caption}
\usepackage{subcaption}

\newtheoremstyle{mythm}
  {\topsep} 
  {\topsep} 
  {\itshape} 
  {} 
  {\bfseries} 
  {.} 
  {.5em} 
  {} 

\theoremstyle{plain}
\newtheorem{theorem}{Theorem}[section]
\newtheorem{lemma}[theorem]{Lemma}

\newtheorem{proposition}[theorem]{Proposition}
\crefname{proposition}{proposition}{proposition}
\newtheorem{corollary}[theorem]{Corollary}

\crefname{claim}{claim}{claims}

\theoremstyle{definition}
\newtheorem{definition}[theorem]{Definition}

\crefname{construction}{Construction}{Constructions}

\newtheorem{remark}[theorem]{Remark}
\newtheorem*{remark*}{Remark}

\makeatletter
\newcommand\footnoteref[1]{\protected@xdef\@thefnmark{\ref{#1}}\@footnotemark}
\makeatother

\newcommand{\RR}{\mathbb{R}}

\newcommand{\CC}{\mathbb{C}}
\newcommand{\C}{\mathbb{C}}
\newcommand{\PP}{\mathbb{P}}
\newcommand{\D}{\mathbb{D}}
\newcommand{\DD}{\mathbb{D}}

\newcommand{\CP}{\CC\PP}
\newcommand{\ZZ}{\mathbb{Z}}
\newcommand{\TT}{\mathbb{T}}
\newcommand{\SSS}{\mathbb{S}}
\renewcommand{\S}{\mathbb{S}}
\newcommand{\calF}{\mathcal{F}}

\newcommand{\X}{{\mathcal{X}}}

\newcommand{\F}{{\mathcal{F}}}
\newcommand{\G}{{\mathcal{G}}}
\newcommand{\R}{{\mathbb{R}}}
\newcommand{\N}{{\mathbb{N}}}
\renewcommand{\d}{{\operatorname{d}}}

\newcommand{\rank}{{\operatorname{rk}}}
\newcommand{\Op}{{\mathcal{O}p}}
\newcommand{\wtd}{\widetilde}

\newcommand{\vol}{{\operatorname{vol}}}

\newcommand{\tois}{\xrightarrow{\sim}}

\newcommand{\Sigmatilde}{\widetilde{\Sigma}}
\newcommand{\phitilde}{\widetilde{\phi}}

\newcommand{\absholonomy}{{$\vert$holonomy$\vert$-like }}
\newcommand{\holonomy}{{holonomy-like }}

\DeclareMathOperator{\Id}{Id}

\DeclareMathOperator{\Int}{Int}
\DeclareMathOperator{\Diff}{Diff}
\DeclareMathOperator{\supp}{supp}

\title{
Existence of conformal symplectic foliations\\ on closed manifolds
}

\author{Fabio Gironella\footnote{Humboldt University, Berlin, Germany. Email: \url{gironelf@math.hu-berlin.de}, \url{fabio.gironella.math@gmail.com}} \and Lauran Toussaint\footnote{Université Libre de Bruxelles, Brussels, Belgium. Email: \url{lauran.toussaint@ulb.be}}}
\date{}

\begin{document}
\maketitle

 \begin{abstract}
 
 We consider the existence of symplectic and conformal symplectic (codimension-$1$) foliations on closed manifolds of dimension $\geq 5$.
 Our main theorem, based on a recent result by Bertelson--Meigniez, states that in dimension at least $7$ any almost contact structure is homotopic to a conformal symplectic foliation.
 In dimension $5$ we construct explicit conformal symplectic foliations on every closed, simply-connected, almost contact manifold, as well as honest symplectic foliations on a large subset of them. 
 Lastly, via round-connected sums, we obtain, on closed manifolds, examples of conformal symplectic foliations which admit a linear deformation to contact structures.
 
\end{abstract}

\section{Introduction}
\label{sec:intro}

Manifolds are often studied through geometric structures living on them. 
This approach requires structure with a suitable balance of "flexibility" and "rigidity".
Roughly speaking a flexible structure is determined purely by algebraic topological data. 
As such, they often exist in large classes of examples but do not reflect many geometric properties of the manifold. 
On the contrary, rigid structures impose more constraints on the ambient space. 
The downside is that the questions of existence and classification become more involved for them.
\\
This phenomenon is nicely illustrated by the tight-overtwisted dichotomy for contact structures. 
Overtwisted contact structures satisfy an $h$-principle \cite{Eli89,BEM15} and are flexible. The classification of tight contact structures is much more intricate, reflecting their rigidity.

In this line of thought we study here (codimension-$1$) foliations  whose leaves are endowed with a symplectic-type geometric structure.
(Without explicit mention of the contrary, the word ``foliation'' will always mean ``codimension $1$ foliation'' in this paper.)
We consider the following variations:

\begin{itemize}
\item A \emph{symplectic foliation} is a foliation $\F$ endowed with a leafwise symplectic form, i.e. $\omega \in \Omega^2(\F)$ which is closed and non-degenerate.

\item A \emph{strong symplectic foliation} is a symplectic foliations $(\F,\omega)$ for which $\omega$ admits an extension to a closed $2$-form on $M$.

\item A \emph{conformal symplectic foliation} $(\F,\eta,\omega)$ on $M^{2n+1}$ consists of a smooth foliation $\F$, endowed with differential forms $\eta \in \Omega^1(\F)$ and $\omega \in \Omega^2(\F)$ satisfying:
\[ \d \eta = 0 , \quad \omega^n > 0, \quad \d_\eta \omega := \d \omega - \eta \wedge \omega = 0.\]
Analogously to the non-foliated case, one considers such pairs up the equivalence relation described by $(\eta,\omega) \sim (\eta + d f, e^f \omega)$ for any $f \in C^\infty(M)$. 
Also, the non-degeneracy of $\omega$ implies that $\eta$ is uniquely defined by the equation $\d \omega = \eta \wedge \omega$.
Hence, strictly speaking it is not an additional datum but we include it for notational convenience.
\end{itemize}

The formal counterpart of all three structures above is an \emph{almost contact structures}, i.e. a pair $(\zeta,\mu)$ consisting of an hyperplane field $\zeta$ endowed with a non-degenerate form $\mu \in \Omega^2(\zeta)$. 
Thus, it is natural to ask which almost contact manifold admit foliations of each of the types above.

Symplectic foliations naturally generalize symplectic structures to the foliated setting. However, their existence is a big open problem in this theory.
For instance, it is not known which spheres $\S^{2n+1}$, $n \geq 3$, admit such foliations, and even the fact that $\SSS^5$ admits one is a non-trivial result \cite{Mit18}.\\
Strong symplectic foliations share the same problem and are in fact even more rigid. For example, their leaf spaces are isomorphic to those of taut foliations in dimension $3$ \cite{MTdPP18}.
(Recall that a foliation is \emph{taut} if it admits a closed transversal intersecting every leaf.)
Notice that Meigniez showed that in higher dimensions taut foliations are flexible \cite{Mei17}. 
Hence, the rigidity of strong symplectic foliations suggests that they could provide the ``right'' generalization of taut foliations to higher dimensions.

The main result of this paper states that an almost contact structures is not only necessary, but also sufficient for the existence of a conformal symplectic foliation in dimension at least $7$:

\begin{theorem}\label{thm:exist_conf_sympl_fol}
Let $M$ be a closed manifold of dimension $2n+1 \geq 7$, endowed with an almost contact structure $(\zeta,\mu)$. 
Then, there exists an exact conformal symplectic foliation $(\F,\eta,\d_\eta\lambda)$ on $M$ homotopic to $(\zeta,\mu)$ through almost contact structures.
\end{theorem}

The main result of \cite{Mei17} implies that any almost contact structure in dimension $\geq 5$ is homotopic to a taut almost symplectic foliation.
As such our main theorem is an immediate consequence of the following:

\begin{proposition}\label{prop:conf_sympl_fol_taut}
Let $n \geq 3$, and $M^{2n+1}$ be a closed manifold with a taut almost symplectic foliation $(\G,\mu)$.
Then, $M$ admits an exact conformal symplectic foliation $(\F,\eta,\d_\eta\lambda)$ such that
\begin{enumerate}[(i)]
\item $(\F,\eta,\d_\eta\lambda)$ is homotopic to $(\G,\mu)$ among almost contact structures.
\item $\F$ is obtained from $\G$ by turbulizing along (nested) hypersurfaces diffeomorphic to $\S^1 \times \S^{2n-1}$.
\end{enumerate}
\end{proposition}

The idea of the proof of \Cref{prop:conf_sympl_fol_taut} is as follows. 
First, we interpret the complement of (tubular neighborhoods of) two transverse curves as a foliated almost symplectic cobordism between almost contact foliations. 
Using the h-principle result from \cite{BerMei}, this cobordism can be homotoped to a conformal symplectic foliation with foliated contact overtwisted boundaries.
Then, we extend the conformal symplectic foliation over the neighborhoods of the transverse curves using the following result: 

\begin{corollary}\label{cor:ConfFoliatedFilling}
Let $\xi_{ot}$ be any overtwisted overtwisted contact structure on $\S^{2n-1}$, $n >2$, in the same almost contact class as the standard tight contact structure. Then, there exists a foliated conformal symplectic cobordism
\[
\emptyset \to \S^1 \times (\S^{2n-1},\xi_{ot}),
\]
where the latter is endowed with the foliation by contact spheres $(\S^{2n-1},\xi_{ot})$.
\end{corollary}

\begin{remark}
\label{rmk:alternative_proof_main_cor}
Alternatively, \Cref{thm:exist_conf_sympl_fol} also follows from \cite[Theorem C]{BerMei} and using \Cref{cor:ConfFoliatedFilling} only once.
As far as existence of conformal symplectic foliations is concerned, the advantage of our approach is the fact that \Cref{prop:conf_sympl_fol_taut} does not require the foliation to be \emph{holonomous} (as defined in \cite[Theorem C]{BerMei}).
\end{remark}

It is well known that a ball does not admit any symplectic structure filling an overtwisted contact structure on the boundary sphere, and hence no conformal symplectic structure either (since the ball is contractible). 
On the other hand, the corollary says that in the foliated setting this is possible.
Its proof is based on the fact that the foliated conformal symplectic cobordism relation between (unimodular) contact foliations is \emph{symmetric} (c.f. \Cref{prop:CobordismEquivalenceRelation}). 
This is in stark contrast with the symplectic case (foliated or not), where the cobordism relation is far from symmetric.\\
The dimensional assumption in the above corollary is essential.
Indeed, its proof uses \cite{EliMur15} to construct a symplectic cobordism from $(\S^{2n-1},\xi_{ot})$ to $(\S^{2n-1},\xi_{st})$. Such a cobordism does not exist for $n=2$, see \cite{MroRol06,KroMroBook}. 
This is also the reason for the dimensional assumption in \Cref{thm:exist_conf_sympl_fol}.

In dimension $5$ we take a different approach and focus on simply connected manifolds. Here, using turbulization of conformal symplectic foliations together with Barden's classification \cite{Bar65} we prove:

\begin{theorem}\label{thm:conf_sympl_fol_5folds}
Any simply connected $5$-manifold admits a conformal symplectic foliation.
\end{theorem}

In many cases, the foliations in \Cref{thm:conf_sympl_fol_5folds} contain closed leaves diffeomorphic to $\S^1 \times \S^3$, and cannot be symplectic.
This being said, using a surgery construction on contact open books from \cite{Eli04}, we prove the following:
\begin{theorem}
\label{thm:sympl_fol_5folds}
Every simply-connected almost contact $5$-manifold $M$ satisfying at least one of the following conditions admits 
symplectic foliation:
\begin{enumerate}[(i)]
    \item $\rank(H^{2}(M;\ZZ))\geq 2$,
    \item\label{item:thm_sympl_fol_5folds_2} $M$ is spin (i.e.\ with trivial second Stiefel--Whitney class $w_2$) and $\rank(H^{2}(M;\ZZ))\leq 1$,
    \item $M$ is spin and $H^2(M;\ZZ) = \ZZ_p \oplus \ZZ_p$ for $p$ relatively prime to $3$,
    \item $M$ is spin and $H^2(M;\ZZ) = \ZZ_q \oplus \ZZ_q$ for $q$ relatively prime to $2$.
\end{enumerate}
\end{theorem}
In fact, manifolds satisfying $(i)$ and the (only) manifold with $\rank = 1$ satisfying $(ii)$, admit infinitely many, pairwise distinct symplectic foliations.
Here, by the rank $\rank(G)$ of a finitely generated $\ZZ-$module $G$ we mean the rank of its free part.

Going back to the general setting of ambient dimension $2n+1\geq5$, we prove a conformal symplectic foliated analogue of an observation due to Milnor \cite[Corollary 6]{Law71}.
This gives an ``explicit'' (at least at the level of underlying smooth foliations) method of constructing conformal symplectic foliations on connected sums with exotic spheres:

\begin{theorem}
\label{thm:conf_sympl_fol_up_to_exotic_sphere}
If $M$ admits a conformal symplectic foliation, then so does the connect sum $M\#\Sigma$ for any exotic sphere $\Sigma$.
The underlying smooth foliation on $M\#\Sigma$ is obtained from the one on $M$ by (iterated) turbulization along a transverse curve, 
and cutting and regluing an open disc in a leaf via an exotic diffeomorphism.
\end{theorem}

The proof uses the conformal symplectic turbulization procedure in order to introduce open leaves which are symplectomorphic to the symplectization of an overtwisted sphere, as well as the results from \cite[Appendix by Courte]{CKS18} that allow to realize any smoothly exotic diffeomorphism of a disk as a compactly supported symplectomorphism of said symplectization.

Lastly, we make a few remarks concerning deformations of conformal symplectic foliations to contact structures, in the spirit of \cite{EliThu98}.
More precisely, we consider \emph{type I} linear contact deformations  (see \Cref{def:TypeIConformalSymplectic}) which are a higher dimensional analogue of the linear deformations from \cite{EliThu98}.

We show that some of the compact leaves contained in the foliation from \Cref{cor:ConfFoliatedFilling}
\emph{obstruct} the existence of a type I deformation (c.f.\ \Cref{lem:obstruction_deformation}). 
As a result, the foliations from \Cref{thm:exist_conf_sympl_fol} cannot be deformed. 
However, if we avoid the use of \Cref{cor:ConfFoliatedFilling}, these problematic contact leaves do not appear, and we obtain examples of (taut) conformal symplectic foliations which admit type I linear contact deformations on closed manifolds.

In order to give the precise statement, recall that given closed curves $\gamma_\pm \subset M_\pm$, the \emph{round-connected sum} $M_- \#_{\S^1} M_+$ is obtained by removing tubular neighborhoods of two (disjoint, parallel) copies of $\gamma_\pm$ inside $M_\pm$ respectively, and gluing along the resulting boundary components.

\begin{theorem}
\label{thm:deformations_self_round_connected_sums}
Let $M^{4n+3}$, $n \geq 1$, be an almost contact manifold, and $\gamma\subset M$ a closed curve. 
Then, there are infinitely many $k>0$ for which the iterated round-connected sum
\[ M \#_{\S^1} (\bigsqcup_{i=1}^k T^2 \times \S^{4n+1}),\]
admits a conformal symplectic foliation having a Type I deformation to contact structures. 
Here, the round-connected sum pairs a parallel copy $\gamma_i$ of $\gamma$ with a curve $\S^1 \times \{pt\}$ in the $i$-th copy of $T^2 \times \S^{4n+1}$.
\end{theorem}

\begin{remark}
\label{rmk:starting_from_fol}
In the above statement, if one already has a taut almost symplectic foliation on $M$, by choosing a transverse curve $\gamma$ intersecting every leaf, the resulting smooth foliation on the round connect sum manifold can be arranged to coincide with the one on $M$ away from the round connect sum region.
\end{remark}

\begin{remark}
\label{rmk:deform_non_taut}
The foliations obtained in \Cref{thm:deformations_self_round_connected_sums} are taut.
However, tautness is not a necessary condition for the existence of Type I deformations. 
For instance, gluing two mapping tori of Liouville domains (after turbulizing at their boundary), yields a conformal symplectic foliation on a closed manifold admitting a Type I deformation \cite[Theorem 2.6.31]{ToussaintThesis}.\\
In fact, we show (c.f.\ \Cref{prop:existence_deformations_contact_away_tubes}) that on the complement of any three curves, two of which are parallel copies of each other, an almost contact structure is homotopic to a \emph{non-taut} exact conformal symplectic foliation admitting a linear deformation to contact structures.
\end{remark}

\subsubsection*{Outline of the paper}

In \Cref{sec:prelim} we give some preliminary statements and considerations.
More precisely, in \Cref{sec:simply-conn-5folds} we recall the results by Smale \cite{Sma62} and Barden \cite{Bar65} which classify closed simply-connected $5-$manifolds.
In \Cref{sec:normal_forms_near_boundary} we give some normal forms for conformal symplectic manifolds near their boundary, which naturally generalize to the foliated setting.
Then, in \Cref{sec:sympl_turbul} we describe the symplectic turbulization and resulting gluing statement for turbulized foliations. 
In \Cref{sec:conf_sympl_cobordisms} we describe how the results from \cite{EliMur15} give homotopies relative to the boundary in the conformal symplectic setting, and recall one of the main results of \cite{BerMei}.

In \Cref{sec:sympl_fol_5folds} we first describe a version of Eliashberg's capping construction for $3$-dimensional open books, that allows us to pass from a $5$-dimensional contact open book to an open book with cosymplectic type boundary. 
We then use this to give a proof of \Cref{thm:sympl_fol_5folds}, i.e.\ we give explicit symplectic foliations on some simply-connected $5$-manifolds.

In \Cref{sec:constructions_conf_sympl_fol} we give some general constructions of conformal symplectic foliations.
More precisely, in \Cref{sec:turbulization_and_gluing} we describe the main procedures we use in order to construct conformal symplectic foliations, namely conformal symplectic turbulization and the resulting gluing statement.
In \Cref{sec:conf_sympl_homot_contact_str} we describe a conformal symplectic foliation on $\S^1\times [0,1]\times M$ which is tangent to the two boundary components and restricts to it to the conformal symplectization of any two contact structures on $M$ which are homotopic as almost contact structures; this will be used in the later sections.
In \Cref{sec:fol_round_handle} we describe how to construct conformal symplectic foliations on round connect sums of certain manifolds, which will be useful in the proof of \Cref{thm:sympl_fol_5folds}.
Lastly, in \Cref{sec:conf_sympl_fol_conn_sum_exotic_spheres} we prove \Cref{thm:conf_sympl_fol_up_to_exotic_sphere}, stating that, on the connected sum of a conformal-symplectically foliated manifold with an exotic sphere, one can construct (explicitly in terms of the original foliation) a conformal symplectic foliation.

In \Cref{sec:proof_conf_sympl_fol_5folds} we prove \Cref{thm:conf_sympl_fol_5folds} on the existence of conformal symplectic foliations on closed simply-connected almost contact $5$-manifolds explicitly given by gluing turbulized symplectic mapping tori.

In \Cref{sec:general_existence} we deal with the general existence statement for conformal symplectic foliations in dimensions $\geq 7$.
More precisely, in \Cref{sec:conf_sympl_fol_cobordisms} we describe some useful conformal symplectic foliations on manifolds with non-empty boundary, which will be used in the sections that follow.
In \Cref{sec:from_alm_sympl_to_conf_sympl} we describe the conformal symplectic Reeb components used in the proof of the main result of \Cref{sec:general_existence}.
In \Cref{sec:proof_h_principle} we prove \Cref{thm:exist_conf_sympl_fol}, stating that the existence of a conformal symplectic foliation in each almost contact class in dimension at least $7$.
Then, in \Cref{sec:proof_main_thm} we prove \Cref{prop:conf_sympl_fol_taut} that states the existence of a conformal symplectic leafwise structure on any taut almost symplectic foliation in dimension $\geq7$, after adding some (homotopically trivial) Reeb components.
Lastly, in \Cref{sec:deformation_to_contact} we collect some observations regarding contact deformations of the conformal symplectic foliations described in the previous sections.
More precisely, we start the subsection by recalling the definition of type I linear deformation, as well as a general existence theorem, and prove \Cref{lem:obstruction_deformation} stating that a model we introduce in the proof of \Cref{thm:exist_conf_sympl_fol} obstruct the existence of type I deformations.
In \Cref{sec:deformation_to_contact_non_taut} we describe how the foliations from \Cref{thm:exist_conf_sympl_fol} on the complement of some codimension $0$ submanifolds give examples of non-taut conformal symplectic foliation on manifolds with boundary which are deformable to contact structures; we also describe how to extend this deformation to the closed manifold as a deformation to hyperplane fields which are contact structures away from three hypersurfaces. 
Lastly, in \Cref{sec:deformation_for_round_connected_sums} we prove \Cref{thm:deformations_self_round_connected_sums}, which gives examples of linearly contact-deformable conformal symplectic foliations on closed manifolds of dimension $4k+3\geq 7$, as well as give details concerning \Cref{rmk:deform_non_taut}, dealing with linearly deformable non-taut conformal symplectic foliations on manifolds with non-empty boundary.

\subsubsection*{Acknowledgements} 

The authors are very grateful to thank Álvaro del Pino for his insight and support.
They also wish to thank Ga\"el Meigniez and Mélanie Bertelson for their useful feedback on a preliminary draft, and in particular for pointing out a mistake in a statement about deformations to contact structures.\\
The first author is supported by the European Research Council (ERC) under the European Union’s Horizon 2020 research and innovation programme (grant agreement No. 772479).
The second author is supported by the F.R.S-FNRS and the FWO under the Excellence of Science programme (grant No. 30950721).

\section{Preliminaries}
\label{sec:prelim}

\subsection{Simply connected $5$-manifolds}
\label{sec:simply-conn-5folds}

The classification of simply connected $5-$manifolds was started by Smale in \cite{Sma62}, who dealt with the case of vanishing second Stiefel--Whitney class, and was completed by Barden in \cite{Bar65}, who covered the general case. Recall that, by the universal coefficient theorem, for a simply connected manifold $M$ we can regard the second Stiefel--Whitney class as a map $w_2(M)\colon H_2(M)\to\ZZ_2$.

\begin{theorem}[\cite{Bar65}]
\label{thm:w2}
Two simply connected $5-$manifolds $M_1$ and $M_2$ are diffeomorphic if and only if there exists an isomorphism of groups $\phi\colon H_2(M_1) \to H_2(M_2)$ preserving the linking product and such that $w_2(M_1) = w_2(M_2)\circ \phi$.
\end{theorem}

As a consequence, one can prove that
any closed (oriented) simply connected $5$-manifold $M$ can be decomposed uniquely into prime manifolds
\[ M = X_j \# M_{k_1} \# \dots \# M_{k_s}\]
with $-1 \leq j \leq \infty$, $s \geq 0$, $k_1$ and $k_i$ dividing $k_{i+1}$ or $k_{i+1} = \infty$. Moreover, the decomposition contains at most one summand of type $X_j$ and possibly no summand of type $M_k$.
The prime manifolds satisfy:

\begin{enumerate}[label=(\roman*)]
    \item $X_{-1}= SU_3/SO_3$ is the Wu-manifold, the homogeneous space obtained from the standard inclusion of $SO_3$ into $SU_3$, and it has $H_2(X_{-1})=\ZZ_2$, $w_2(X_{-1})\neq 0$ and $W_3(X_{-1})\neq 0$, where $W_3$ is the third integral Stiefel-Whitney class.
    \item $X_0=\SSS^5$, having $H_2(\SSS^5)=0$, $w_2(\SSS^5)=0$ and $W_3(\SSS^5)= 0$.
    Furthermore, $X_0$ is only needed in the decomposition if $M\simeq \SSS^5$.
    \item $X_\infty=\SSS^2\overset{\sim}{\times}_\gamma \SSS^3$ the total space of the non-trivial $\SSS^3-$bundle over $\SSS^2$, which has $H_2(X_\infty)=\ZZ$, $w_2(X_\infty)\neq0$ and $W_3(X_\infty)= 0$.
    \item $M_\infty = \SSS^2 \times \SSS^3$ which has $H_2(M_\infty)=\ZZ$, $w_2(M_\infty)=0$ and $W_3(M_\infty)= 0$.
    \item for $1<k<\infty$, $M_k$ has $H_2(M_k)=\ZZ_k\oplus \ZZ_k$, $w_2(M_k)=0$ and $W_3(M_k)= 0$.
    \item for $1\leq j <\infty$, $X_j$ has $H_2(X_j)=\ZZ_{2^j}\oplus \ZZ_{2^j}$, $w_2(X_j)=0$ and $W_3(X_j)\neq 0$.
\end{enumerate}

A (conformal) symplectic foliated manifold in particular admits a (corank-$1$) distribution endowed with a non-degenerate $2$-form. 
That is, the prime manifolds which are relevant to us admit almost contact structures.

\begin{lemma}[{\cite[Lemma 7]{Gei91}}]
\label{lemma:simpl_conn_5fold_contact}
A simply connected $5-$manifold admits an almost contact structure if and only if its third integral Stiefel--Whitney class $W_3$ vanishes.
\end{lemma}

As $w_2$ is additive under connected sums and $W_3$ vanishes if and only if $w_2$ is in the image of the reduction mod $2$ morphism $H_2(\cdot,\ZZ)\to H_2(\cdot,\ZZ/2\ZZ)$, according to \cite{Bar65}, any simply connected $5-$manifold which admits a symplectic foliation is hence of the form
\begin{equation}
\label{eq:possibilities_dim5}
M=\SSS^5, \text{ or } M = M_{k_1}\# \ldots \# M_{k_s}, \text{ or }
M = X_\infty \# M_{k_1} \# \ldots \# M_{k_s} ,
\end{equation}
with $s \geq 0$, $k_1$ and $k_i$ dividing $k_{i+1}$ or $k_{i+1} = \infty$.
Moreover, a complete set of invariants for manifolds in this class are the homology group $H_2$ and the fact whether $w_2$ is trivial or not.

\subsection{Normal form around the boundary}
\label{sec:normal_forms_near_boundary}

Let $(M^{2n},\eta,\omega)$ be a conformal symplectic manifold with boundary $\partial M$, and denote by $(\eta_\partial,\omega_\partial)$ the restriction of the forms to the boundary.
Notice that $\omega_\partial$ is $\d_{\eta_\partial}$-closed, and that it has one-dimensional kernel. 
Moreover, one can always find an \emph{admissible form} for $\omega_\partial$, by which we mean  a $\beta \in \Omega^1(\partial M)$ such that
\begin{equation}\label{eq:admissibleform}
    \omega_\partial^{n-1} \wedge \beta > 0.
\end{equation}
For instance, for any choice of vector field $X$ on $M$ transverse to the boundary, one can take $\beta := \iota_X\omega|_{\partial M}$.
Note that the converse is also true: a choice of admissible form determines a vector field transverse to the boundary.

\begin{theorem}\label{thm:ConformalSymplecticNormalform}
Let $(M,\eta,\omega)$ be a conformal symplectic manifold, with $\partial M \neq \emptyset$.
Let also $\beta \in \Omega^1(\partial M)$ be any admissible form for $(\eta_\partial,\omega_\partial)$.
Then, a neighborhood of the boundary is conformally equivalent to:
\[ 
\Big( (-\varepsilon,0] \times \partial M,\, \eta = \eta_\partial,\, \omega = \omega_\partial + \d_{\eta_{\partial}}(t \beta)\Big).
\]
\end{theorem}

The following elementary observation will be used in the proof.
\begin{lemma}\label{lem:conformalexact_trick}
Let $\eta$ be a closed $1$-form on $M$, such that $\ker \eta \pitchfork \partial M$. If $\alpha \in \Omega^k(M)$ satisfies
\[ \d_\eta \alpha = 0, \quad \alpha|_{\partial M} = 0,\]
then, locally around the boundary, $\alpha$ is $\d_\eta$-exact.
\end{lemma}
\begin{proof}[Proof of \Cref{lem:conformalexact_trick}]
Fix a collar neighborhood $(-\varepsilon,0]\times \partial M$ of the boundary. In these coordinates we can write
\[ \alpha = \alpha_t + \d t \wedge \beta_t,\]
for $\alpha_t \in \Omega^k{\partial M}$, $\beta_t \in \Omega^{k-1}(\partial M)$, and $t \in (-\varepsilon,0]$. 

We may assume $\partial_t \in \ker \eta$ which implies that $\eta = \eta_\partial$ does not have a $\d t$ component. Then, it follows from $\d_\eta \alpha =0$ and the above equation that $\alpha_0 = 0$, $\overline{\d}_\eta \alpha_t = 0$, and $\dot{\alpha}_t = \overline{\d}_\eta \beta_t$.
Therefore we have
\[ \alpha_t = \alpha_t -\alpha_0 =  
\int_0^t \dot{\alpha}_s \d s =
\overline{\d}_\eta \left( \int_0^t \beta_s \d s\right),
\]
which in turn implies 
\[ 
\alpha = 
\overline{\d}_\eta \left( \int_0^t \beta_s \d s\right) + \d t\wedge \beta_t =  
\d_\eta \left( \int_0^t \beta_s \d s\right).
\qedhere
\]

\end{proof}

\begin{proof}[Proof of Theorem \ref{thm:ConformalSymplecticNormalform}]
Fix a collar neighborhood $(-\varepsilon,0] \times \partial M$ of the boundary. In these coordinates we define
\[ \wtd{\eta} := \eta_\partial,\quad \wtd{\omega} := \omega_\partial + \d_{\wtd{\eta}}(t\beta),\]
which is easily checked to be a conformal symplectic structure for $\varepsilon>0$ small enough. 
Since $\eta - \wtd{\eta}$ is a $\d$-closed form whose restriction to $\partial M$ vanishes, it follows from Lemma \ref{lem:conformalexact_trick}
that
\[\eta = \wtd{\eta} + \d f,\]
for some function $f$ satisfying $f|_{\partial M}  = 0$.
Note that $(\eta,\omega)$ is equivalent to $(\wtd{\eta},e^{-f}\omega)$. 
Similarly to before, as $e^{-f}\omega - \wtd{\omega}$ is $\d_{\wtd{\eta}}$-closed and vanishes on $\partial M$, \Cref{lem:conformalexact_trick} implies that
\[ \wtd{\omega} = e^{-f}\omega + \d_{\eta_{\partial}} \beta,\]
for some $\beta \in \Omega^1(M)$.
As such, the $1$-parameter family of conformal symplectic structures
\[ \eta_s:= \eta_\partial,\quad \omega_s := s e^{-f} \omega + (1-s) \wtd{\omega} = e^{-f}\omega+ (1-s) \d_{\eta_\partial} \beta,\]
satisfies the conditions conditions of  \cite[Theorem 4]{Ban02} (or \cite[Corollary 3.2]{BanKot09}). 
Thus there exists an isotopy $\phi_s$ and a family of functions $g_s$ such that $\phi_s^*\omega_s = g_s \omega_0$ (and hence $\phi_s^*\eta_s = \eta_0 + d g_s$). 
In particular, $\phi_1$ is the desired conformal symplectomorphism.
\end{proof}

Just as in the symplectic setting we distinguish special types of boundaries.
\begin{definition}\label{def:ConformalContactType}
A conformal symplectic manifold $(M,\eta,\omega)$ is said to have  \emph{convex (resp.\ concave) boundary of contact type} if there exists a vector field $X$ defined on $\Op(\partial M)$, transverse to $\partial M$, pointing outwards (resp.\ pointing inwards) and which is \emph{$\eta-$Liouville} for $\omega$, i.e.\ satisfies
\[ 
\omega = \d_\eta \iota_X \omega.
\]
\end{definition}

Of course, when $\eta = 0$ this agrees with the usual definition for symplectic manifolds. 
An immediate corollary of the normal form of \Cref{thm:ConformalSymplecticNormalform} is that being of contact type can be detected on the boundary (as in the honest symplectic case):

\begin{lemma}\label{lem:contacttypeboundaryequivalence}
Let $(M,\eta,\omega)$ be a conformal symplectic manifold with boundary. Then the following are equivalent:
\begin{enumerate}[(i)]
    \item $(\eta,\omega)$ is of convex/concave contact type at the boundary; 
    \item There exists a vector field $X$ on $M$ transverse to $\partial M$,  pointing outwards/inwards such that for all $p \in \partial M$
    \[ \omega_p = (\d_\eta \iota_X\omega)_p,\]
    \item There exists an admissible form $\alpha \in \Omega^1(\partial M)$ for $\omega_\partial := \omega|_{\partial M}$ satisfying
    \[ \omega_\partial =  \pm\d_{\eta_\partial}\alpha.\]
\end{enumerate}
\end{lemma}

For boundaries of contact type it is often useful to slightly rewrite the normal form from  \Cref{thm:ConformalSymplecticNormalform}. 
We state it here as a corollary for later reference.
\begin{corollary}
\label{cor:ConformalContactTypeNormalform}
Let $(M,\eta,\omega)$ be a conformal symplectic manifold with contact type boundary. 
Then, locally around the boundary, it is conformally equivalent to
\[ 
\Big( (-\varepsilon,0] \times \partial M,\eta = 
\eta_\partial, \omega = \d_{\eta_\partial}( e^{\pm t} \alpha)\Big),
\]
where $\alpha \in \Omega^1(M)$ is an admissible form such that $\omega_\partial = \d_{\eta_\partial} \alpha$. The sign is a plus (resp. minus) in the convex (resp. concave) boundary case.
\end{corollary}

The entire discussion above carries over to the foliated case.
Consider a conformal symplectic foliation $(M,\F,\eta,\omega)$ transverse to the boundary, and denote by $(F_\partial,\eta_\partial,\omega_\partial)$ its restriction to $\partial M$. 
As in Equation \ref{eq:admissibleform}, we define an \emph{admissible form for $\omega_\partial$} to be $\alpha \in \Omega^1(\F_\partial)$ such that
\begin{equation}\label{eq:ConformalAdmissibleForm}
    \alpha \wedge \omega_\partial^{n-1} > 0.
\end{equation}

The key point in the proof of \Cref{thm:ConformalSymplecticNormalform} is the Moser trick which produces the desired diffeomorphism as the flow of a vector field. The equation defining this vector field can be solved leafwise, that is, in the foliated case the vector field can be chosen tangent to the leaves. Thus, the Moser trick can be carried out leafwise, giving the following normal form.

\begin{theorem}
\label{thm:FoliatedConformalSymplecticNormalform}
Let $(\F,\eta,\omega)$ be a conformal symplectic foliation on $M$ transverse to the boundary. Denote by $(\F_\partial,\eta_\partial,\omega_\partial)$ the induced conformal symplectic foliation on the boundary. Then there exists a collar neighborhood $(-\varepsilon,0]\times \partial M$ on which
\[ \F = (-\varepsilon,0] \times \F_{\partial},\quad \eta = \eta_\partial,\quad \omega = \omega_\partial + \d_{\eta_\partial}(t \alpha).\]
\end{theorem}
Here, the notation $(-\varepsilon,0] \times \F_\partial$ refers to taking the product foliation:
\[ \bigcup_{\mathcal{L}_\partial \in \F_\partial} (-\varepsilon,0] \times \mathcal{L}_\partial,\]
where the union is over the leaves $\mathcal{L}_\partial$ of $\F_\partial$. 
According to our usual conventions, this means that if $\F_\partial = \ker \gamma_\partial$ and $\mu_\partial \in \Omega^{n-1}(\F_\partial)$ is a leafwise positive volume form, then $\gamma_\partial \wedge \d t \wedge \mu_\partial$ is a positive volume form on the collar neighborhood from the theorem.

We also point out that, when $\eta_\partial = 0$ in the above statement, we obtain the usual normal form for symplectic foliations.

\subsection{Symplectic gluing and turbulization}
\label{sec:sympl_turbul}

Constructing foliations often involves gluing foliated manifolds with boundary. 
As such it is useful to change foliations transverse to the boundary into ones tangent to the boundary. 
\\
Let $\F$ be a foliation on $M$ transverse to the boundary. Then there exists a collar neighborhood $(-\varepsilon,0] \times \partial M$ on which $\F$ is a product foliation, i.e. $\F = \ker \gamma_\partial$ with $\gamma_\partial \in \Omega^1(\partial M)$. To turn $\F$ into a foliation tangent to the boundary we need to assume that $\gamma_\partial $ is closed. 
Choose a smooth function $f:(-\varepsilon,0] \to [0,1]$ which is zero near $-\varepsilon$, $f(0) = 1$ and such that it can be smoothly extended as the constant function on $[0,\varepsilon)$.
Then, in the collar we can also define
\[ \wtd{\gamma} = (1-f(t)) \gamma_\partial + f(t) \d t.\]
We say that $\wtd{\F} := \ker \wtd{\gamma}$ is obtained from $\wtd{\F}$ by \emph{turbulization}. Note that $\wtd{\F}$ is tangent to the boundary and agrees with $\F$ away from the boundary. The condition that $f$ can be smoothly extended ensures that the gluing of two turbulized foliations is again smooth.

For symplectic foliations we also need control over the leafwise structure to ensure that the gluing is smooth. To make this precise consider a symplectic foliation $(\F,\omega)$ on $M$ tangent to the boundary. Choose a collar neighborhood of the boundary $k: (-\varepsilon,0] \times \partial M \to M$, and use it to define
\[ M_\infty := M \cup_{\partial M} [0,\infty) \times \partial M.\]
On $[0,\infty) \times \partial M$ we define an (a priori only continuous) extension of $(\F,\omega)$ by:
\[ \F' \coloneqq \bigcup_{t \in [0,\infty)} \{t \} \times \partial M,\quad \omega' \coloneqq \omega_\partial.\]
If this extension is smooth we say that the collar neighborhood is \emph{adapted}. 

\begin{definition}\label{def:tame_boundary}
A symplectic foliation is said to be \emph{tame at the boundary} if it admits an adapted collar neighborhood as above.
\end{definition}
It follows immediately that such manifolds can be glued by "matching admissible collars". To be precise we have:
\begin{proposition}
\label{thm:sympl_gluing_turbul}
Let $(M_i,\F_i,\omega_i)$, $i=1,2$ be symplectic foliations tame at the boundary. 
If there exists an orientation reversing diffeomorphism $\phi:\partial M_1\to \partial M_2$ satisfying $\phi^*\omega_{2,\partial} = \omega_{1,\partial}$, then $M_1\cup_\phi M_2$ admits a symplectic foliation.
\end{proposition}
 Note that this means that the orientation on the boundary is not the one induced by the symplectic form.

To state the symplectic version of the turbulization construction we need one more definition: $(\F,\omega)$ is said to have \emph{boundary of cosymplectic type} if the restriction at the boundary $(\calF\vert_{\partial},\omega_\partial)$ admits a leafwise admissible form (see Equation \ref{eq:ConformalAdmissibleForm}) which is additionally leafwise closed; 
we will also call the latter an \emph{admissible form of cosymplectic type}. 
If additionally $\F_\partial$ is unimodular (i.e. defined by a closed $1$-form) then we say $(\F,\omega)$ has \emph{boundary of unimodular cosymplectic type}.

\begin{theorem}[{\cite[Section 1.7]{ToussaintThesis}}] 
\label{thm:turbulization}
Let $(M,\F,\omega)$ be a symplectic foliation with boundary of unimodular cosymplectic type.
Then $M$ admits a symplectic foliation $(\wtd{\F},\wtd{\omega})$ tame at the boundary.
Moreover:
\begin{enumerate}
\item $(\wtd{\F},\wtd{\omega})$ and $(\F,\omega)$ agree away from the boundary;
\item Let $\gamma_\partial$ be any closed form defining $\F_\partial$, and $\alpha$ any admissible form of cosymplectic type for $\omega_\partial$.
Then we can arrange that 
\[ \wtd{\omega}_\partial = \omega_\partial + \gamma_\partial \wedge \alpha,\]
on the boundary leaf $\partial M$.
\end{enumerate}
\end{theorem}

One can apply this theorem to suitable decompositions of a manifold (as explicitly stated in \Cref{prop:from_OBD_to_sympl_fol} below) in order to obtain symplectic foliations.
In order to describe this, first recall that an \emph{abstract open book} is a pair $(\Sigma,\phi)$ consisting of a manifold with boundary $\Sigma$, and a diffeomorphism $\phi: \Sigma \tois \Sigma$ which is the identity near the boundary.
Out of this data we can construct a manifold $M_{(\Sigma,\phi)}$ with a natural decomposition
\[ M_{(\Sigma,\phi)} := (B \times \D^2) \cup_{B\times \S^1} \Sigma_\phi,\]
where $B:= \partial \Sigma$ and $\Sigma_\phi$ denotes the mapping torus associated to $\phi$, i.e. $\Sigma \times \R/\sim$ where $(x,t) \sim (\phi(x), t-1)$. Usually we refer to $B\times \D^2$, and $\Sigma_\phi$ respectively as the \emph{inside and outside components} of the open book.

\begin{definition}\label{def:obdtypes}
An abstract open book $(\Sigma,\phi)$ is said to be of:
\begin{enumerate}
    \item \emph{contact type} if $(\Sigma,\d \lambda)$ is a Liouville domain, and $\phi$ is an exact symplectomorphism;
    \item \emph{cosymplectic type} if $(\Sigma,\omega)$ is a symplectic manifold with boundary of cosymplectic type, and $\phi$ is a symplectomorphism.
\end{enumerate}
\end{definition}

It is well-known since \cite{Gir02} 
that if $(\Sigma,\d \lambda,\phi)$ is of contact type, then $M_{(\Sigma,\phi)}$ admits a contact structure supported by the open book. 
The analogue for symplectic foliations using open books of cosymplectic type follows directly from Theorem \ref{thm:turbulization}:

\begin{proposition}
\label{prop:from_OBD_to_sympl_fol}
Let $(\Sigma,\omega)$ be a $2n-$dimensional symplectic manifold with cosymplectic boundary $B$ and $\phi:\Sigma \to \Sigma$ a symplectomorphism which is the identity near the boundary. 
Then, the  open book $M_{(\Sigma,\phi)}$ admits a symplectic foliation.
\end{proposition}

\begin{proof}
The proposition follows from applying Theorem \ref{thm:turbulization} to both components of the open book. 
On the outside component $\Sigma_\phi$ we take the symplectic foliation $(\F := \ker \d \theta,\omega)$ where $\d \theta$ denotes the pullback of the angular form on $\S^1$ under the projection $\pi:\Sigma_\phi \to \S^1$. 
After turbulization the boundary leaf equals
\[ (B \times \S^1, \omega_\partial + \d \theta \wedge \alpha),\]
where $\alpha\in\Omega^1(B)$ is any admissible form of cosymplectic type for $\omega_\partial$. 

Similarly, on the inside component $B\times \D^2$ we take the symplectic foliation $(\F := \ker\alpha, \omega := \omega_\partial + 2r\d r \wedge \d \theta)$, where $(r,\theta) \in \D^2$ denote polar coordinates. Again this satisfies the conditions in Theorem \ref{thm:turbulization} and the resulting boundary leaf equals
\[ (B \times \S^1, \omega_\partial + \alpha \wedge \d \theta).\]
Then $\phi:B \times \S^1 \to B \times \S^1$ defined by $(x,\theta) \mapsto (x,-\theta)$ is an orientation reversing diffeomorphism between the boundaries preserving the symplectic forms. Thus both pieces can be glued using \Cref{thm:gluing_turbul}.
\end{proof}

In most cases the page of a symplectic open book does not have boundary of cosymplectic type. 
For instance, this requires in particular the binding to admit a cosymplectic structure
to which there are topological obstructions; recall that a \emph{cosymplectic structure} on a manifold $M^{2n+1}$ is a pair $(\gamma,\mu)$ where $\gamma \in \Omega^1(M)$ and $\mu \in \Omega^2(M)$ are closed forms satisfying
\[ \gamma \wedge \eta^n > 0.\]

However, in some cases we can perturb the symplectic form on the page to one which has cosymplectic boundary (c.f.\ \cite[Lemma 6.4.6]{TorresThesis}): 
\begin{lemma}\label{lem:cosymplecticboundaryperturbation}
Let $(\Sigma,\omega)$ be a symplectic manifold, and $(\gamma,\mu)$ a cosymplectic structure on $\partial M$ such that:
\begin{enumerate}
    \item\label{item:lem_cosymplecticboundaryperturbation_1} $\gamma \wedge \omega_\partial = 0$;
    \item\label{item:lem_cosymplecticboundaryperturbation_2} $[\mu]$ is in the image of the restriction $H^2(\Sigma) \to H^2(\partial \Sigma)$.
\end{enumerate}
Then there exists a symplectic form $\wtd{\omega}$ on $\Sigma$ with boundary of cosymplectic type. 
\end{lemma}

\begin{proof}
Because of Hypothesis \ref{item:lem_cosymplecticboundaryperturbation_2}, there is a closed extension $\wtd{\mu} \in \Omega^2(\Sigma)$ of $\mu$. 
Then, for $\varepsilon > 0$ sufficiently small,
\[ \wtd{\omega} := \omega + \varepsilon \wtd{\mu},\]
defines a symplectic form on $\Sigma$. 
Lastly, by Hypothesis \ref{item:lem_cosymplecticboundaryperturbation_1}, $\gamma$ is a closed admissible form for $\wtd{\omega}_\partial = \omega_\partial + \varepsilon \mu$, which implies that $\wtd{\omega}$ has boundary of cosymplectic type.
\end{proof}

\subsection{Conformal symplectic cobordisms}
\label{sec:conf_sympl_cobordisms}

To construct (conformal) symplectic cobordisms we will often use the h-principle results from \cite{EliMur15,BerMei}. The formal data underlying such a cobordism consists of the following:

\begin{definition}\label{def:ALmostSymplecticCobordism}
An almost symplectic cobordism $(W,\omega):(M_-,\alpha_-,\omega_-) \to (M_+,\alpha_+,\omega_+)$ between (positive) almost contact manifolds $(M_\pm,\alpha_\pm,\omega_\pm)$ consists of
\begin{enumerate}[(i)]
    \item A smooth oriented cobordism $W$ from $M_-$ to $M_+$;
    \item An almost symplectic structure $\omega$ such that $\omega|_{\partial_\pm W} = \omega_\pm$.
\end{enumerate}
\end{definition}

By a homotopy of almost symplectic cobordisms, we mean a homotopy $\omega_{s}$, $ s\in [0,1]$ of $\omega$. This induces a homotopy of almost contact forms on the boundary. 
In case $\xi_\pm := \ker \alpha_\pm$ define honest contact structures, we often consider homotopies relative to the boundary. These come in two flavours:
\begin{definition}\label{def:CobordismHomotopy}
A homotopy of almost symplectic cobordisms $\omega_s$ is said to be:
\end{definition}
\begin{enumerate}[(i)]
\item weakly relative to the boundary if $\omega_s|_{\partial_\pm W} = f_{\pm,s} \d \alpha_\pm$, for positive functions $f_{\pm,s}\in C^\infty(M_\pm)$. 
\item relative to the boundary if $\omega_s|_{\partial_\pm W} = \d \alpha_\pm$. 
\end{enumerate}

Thus, a homotopy relative to the boundary preserves the contact form while a weakly relative homotopy only preserves the contact structure. 

\begin{theorem}[{\cite[Theorem 1.1]{EliMur15}}]
\label{thm:EliMurphyMain}
Let $(W^{2n},\Omega)$, $n \geq 2$, be an almost symplectic cobordism between (non-empty) contact manifolds $(M_\pm,\xi_\pm)$. 
Assume that $\xi_0$ is overtwisted. If $n=2$, additionally assume that $\xi_1$ is overtwisted. 
Then, $\omega$ is homotopic, weakly relative to $\partial W$, to an (exact) symplectic structure.
\end{theorem}

\begin{remark}
\label{rmk:weak_relat_to_relat}
For homotopies of symplectic structures, being relative or weakly relative to the boundary are rather different conditions.
For example, the above theorem cannot be strengthened to produce homotopies relative to the boundary. Indeed, by Stokes theorem, this would allow us to obtain symplectic cobordisms with negative volume.
For conformal symplectic structures no such obstruction exist. In fact, up to conformal equivalence, any homotopy weakly relative to the boundary is relative to the boundary.
\end{remark}
 Together with \Cref{thm:EliMurphyMain} this implies:

\begin{theorem}[\cite{EliMur15}]\label{thm:MakingCobordismSymplectic}
Let $(W^{2n},\omega)$ be an almost symplectic cobordism
\[ (M_-,\alpha_-) \sqcup (\S^{2n-1},\alpha_{ot}) \to (M_+,\alpha_+),\]
where $\xi_{ot}=\ker\alpha_{ot}$ is any overtwisted contact structure in the almost contact class of $(\S^{2n-1},\xi_{st})$, and $M_+$ is non-empty. If $n=2$ additionally assume $\xi_+=\ker\alpha_+$ is overtwisted. 
Then, $\omega$ is homotopic, relative to the boundary, to a conformal symplectic structure on $W$. 
\end{theorem}

Using foliated Morse theory Bertelson and Meigniez extended \Cref{thm:EliMurphyMain} to the foliated setting.
For their statement recall that a foliation is taut if three is a transverse loop intersecting every leaf.

\begin{theorem}[{\cite[Theorem B]{BerMei}}]
\label{thm:ber_mei}
Let $(M,\F,\omega)$ be a taut, almost symplectic foliation, and $\eta\in \Omega^1(\calF)$ any leafwise closed form. 
Suppose that $\F$ is transverse to the boundary, and $\partial M$ splits as the disjoint union of two non-empty compact subsets $\partial_\pm M$, each intersecting every leaf of $\F$. 
Then there exists $\lambda \in \Omega^1(\F)$ such that
\begin{enumerate}[(i)]
    \item $\d_\eta \lambda$ is leafwise non-degenerate, and homotopic (among non-degenerate $2$-forms) to $\omega$;
    \item $\lambda$ restricts to a positive (resp. negative) overtwisted contact structure on every leaf of $\F|_{\partial_+ M}$ (resp $\F|_{\partial_- M}$).
\end{enumerate}
\end{theorem}

Recall that the foliated analogue of an overtwisted disk is an overtwisted basis. That is, a collection of foliated contact embeddings 
\[ h_i:[0,1] \times B^{2n}_{ot} \to (M,\F),\quad i=1,\dots,N\]
where $B$ denotes a $2n$-dimensional overtwisted ball, such that each leaf of $\F$ intersects at least one of the embeddings. 

In the previous theorem, the homotopy $\d_\eta\lambda$ to $\omega$ induces a homotopy of almost contact foliations on $\partial W$. An inspection of the proof in \cite{BerMei} shows that in some cases these almost contact foliations may assumed to be overtwised. More precisely:

\begin{corollary}\label{cor:BMRelative}
Let $(M,\F,\omega)$ satisfy the conditions of \Cref{thm:ber_mei}. 
Suppose $(\F,\omega)$ defines an overtwisted almost contact foliation 
$(\F_\partial,\alpha_\partial,\omega_\partial)$ on $\partial W$.
Then, the induced homotopy from $(\F_\partial,\lambda|_{\partial W})$ to  $(\F_\partial,\alpha_\partial,\omega_\partial)$ is through overtwisted almost contact foliations.
More precisely, these overtwisted almost contact foliations, seen as codimension $2$ almost contact foliations on $\partial M \times [0,1]_s$  (with $[0,1]_s$ parameter of homotopy) admit an overtwisted basis.
\end{corollary}

\section{Symplectic foliations on (some) simply-connected $5$-manifolds}
\label{sec:sympl_fol_5folds}

The usefulness of Proposition \ref{prop:from_OBD_to_sympl_fol} relies on producing symplectic abstract open books with boundary of cosymplectic type.
One can for instance obtain them by modifying open books supporting contact structures via \Cref{lem:cosymplecticboundaryperturbation}. 
We now describe another explicit procedure to obtain an open book of cosymplectic type from one of contact type.

Recall that, according to \cite{Gir02}, any contact manifold admits a supporting open book whose page is a Weinstein domain;
in particular, one can represent any contact manifold via abstract open books of contact type (as in Definition \ref{def:obdtypes}).
By attaching a symplectic handle to $\Sigma$, we can change its boundary type from contact to cosymplectic as desired. 
More precisely, since the boundary $\partial \Sigma$ is of contact type, it admits a supporting open book decomposition. 
Its outside component is a symplectic fibration over the circle. 
There is then a symplectic cobordism which, at the boundary, corresponds to removing the inside component of the open book and symplectically capping off its pages, yielding as new boundary a symplectic fibration over $\S^1$, which is in particular a cosymplectic structure. 

As we are interested in this paper in the $5-$dimensional case (i.e. $\dim \Sigma = 4$), we give here details only for this dimension, using a (partial) capping construction for $3-$dimensional contact manifolds to Eliashberg \cite{Eli04}.
More precisely, as $\partial \Sigma$ is a $3$-dimensional contact manifold in this case, 
by positive stabilization one can find a supporting open book with binding $\S^1$. 
The above construction then simply amounts to adding a $2$-handle to $\Sigma$.
(The details for the analogous construction in higher dimensions can be found in \cite{DMG14}.)

\begin{proposition}[\cite{Eli04}]
\label{thm:EliashbergTrick}
Let $(\Sigma^4,\d\lambda,\phi)$ be an abstract open book of contact type. 
Then there exists an abstract open book of cosymplectic type $(\wtd{\Sigma},\wtd{\omega},\wtd{\phi})$, 
where
\[ \wtd{\Sigma} := \Sigma \cup_{\S^1 \times \D^2} \D^2 \times \D^2,\]
is obtained from $\Sigma$ by attaching a symplectic $2$-handle, and $\phi$ is extended as the identity.
\end{proposition}

In other words, one can associate 
to each contact $5$-manifold $(M,\xi)$ (or more precisely, to each open book supporting $\xi$) a different $5$-manifold $\wtd{M}$ which admits a symplectic foliation. 
In general, there is no clear way to identify $\wtd{M}$ other than being the manifold obtained from $M$ by surgery along a curve $\S^1 \subset M$.
However, if $M$ is simply connected one also describe $\wtd{M}$ in terms of the classification of Section \ref{sec:simply-conn-5folds}:

\begin{proposition}\label{thm:SFonconnectedsumwithspheres}
Let $(M,\xi)$ be a simply connected contact $5$-manifold. 
Then, $M \# \, \S^2\times\S^3$ admits infinitely many, pairwise smoothly distinct, symplectic foliations.
\end{proposition}

\begin{proof}
It is enough to combine \Cref{prop:from_OBD_to_sympl_fol}, \Cref{thm:EliashbergTrick}, applied to some contact open book and infinitely many of its positive stabilizations, and \Cref{lem:EliashbergTrickClassification} below.
\end{proof}

\begin{remark}
In fact, $M \# \S^3 \times \S^2$ as in the above theorem admits infinitely many pairwise non-isomorphic symplectic foliations. 
Eliashberg's argument \cite{Eli04} shows that $\partial \wtd{\Sigma}$ is a symplectic fibration over $\S^1$ whose fiber is the closed surface obtained by capping the page of an open book decomposition of $\partial \Sigma$. 
That is, the fiber equals $\Sigma_{g}$, the genus-$g$ surface, for some $g \in \N$.
As positive stabilizations yield open books supporting the same contact structure on $\partial \Sigma$, one can also achieve as $\partial \Sigmatilde$ a symplectic fibration with fiber the closed surface $\Sigma_{\wtd{g}}$ for any $\wtd{g}\geq g_0$, where $g_0$ the genus of the page of the original open book decomposition chosen on $\partial \Sigma$. 
In particular, different choices of $\wtd{g}$ yield non-isomorphic boundaries of $\wtd{\Sigma}$.
Therefore, since the symplectic foliation constructed by Proposition \ref{prop:from_OBD_to_sympl_fol} has a single compact leaf diffeomorphic to $\partial \wtd{\Sigma} \times \S^1$, the resulting foliations will be non-isomorphic for different choices of $\wtd{g}$.
\end{remark}

\begin{lemma}\label{lem:EliashbergTrickClassification}
Suppose $M_{(\Sigma,\phi)}$ is simply connected, and $(\wtd{\Sigma},\wtd{\phi})$ be obtained from $(\Sigma,\phi)$ by attaching a $2$-handle and extending the monodromy as $\Id$. 
Then
\[ 
\pi_1(M_{(\wtd{\Sigma},\wtd{\phi})}) = 0,
\quad H_2(M_{(\wtd{\Sigma},\wtd{\phi})}) =
H_2(M_{(\Sigma,\phi)}) \oplus \ZZ, 
\quad w_2(M_{(\wtd{\Sigma},\wtd{\phi})}) =
w_2(M_{(\Sigma,\phi)}).
\]
That is, by Smale-Barden's classification,
\[ 
M_{(\wtd{\Sigma},\wtd{\phi})} = M_{(\Sigma,\phi)} \# \S^2 \times \S^3.
\]
\end{lemma}

\begin{proof}
Observe that the boundary of $\wtd{\Sigma}$ is obtained from that of $\Sigma$ by doing a surgery:
\[ \partial \wtd{\Sigma} = \left( \partial \Sigma \setminus \S^1 \times \DD^2\right) \cup \DD^2 \times \S^1.\]
Using also that $\wtd{\phi}$ is the identity on the $2$-handle we obtain:
\begin{align*}
    M_{(\wtd{\Sigma},\wtd{\phi})} & =
    \partial \wtd{\Sigma}\times \DD^2 \cup \Sigmatilde_{\phitilde}
    \\ & =
    \left((\partial \Sigma\setminus \S^1\times\DD^2)\cup \DD^2\times\S^1\right)\times \DD^2 \cup (\Sigma \cup \DD^2\times\DD^2)_{\phitilde}
    \\ &=
    (\partial \Sigma\times \DD^2 \cup \Sigma_\phi \setminus \S^1\times \DD^2\times \DD^2)\cup (\DD^2\times \S^1\times \DD^2 \cup \DD^2\times \DD^2 \times \S^1)
    \\ & =
    M_{\Sigma,\phi}\setminus \S^1\times \DD^2 \times \DD^2 \cup \DD^2\times (\S^1\times \DD^2\cup \DD^2\times \S^1)
    \\ & =
    M_{\Sigma,\phi}\setminus \S^1\times \DD^4 \cup \DD^2\times \S^3,
\end{align*}
where each of the gluing maps is the identity.
That is, $M_{(\wtd{\Sigma},\wtd{\phi})}$ is obtained from $M_{(\Sigma,\phi)}$ by surgery along an $\S^1$.
This proves the claims about the fundamental and second homology groups.

By Mayer-Vietoris' theorem,
\[H^2(M_{(\Sigma,\phi)}) = H^2(M_{(\Sigma,\phi)} \setminus \S^1 \times \DD^4),\quad \text{and,}\quad 0 \to \ZZ \to H^2(M_{(\wtd{\Sigma},\wtd{\phi})}) \to H^2(M_{(\Sigma,\phi)} \setminus \S^1 \times \DD^4) \to 0.\]
Hence, by naturality of $w_2$, we see that if $w_2(M_{(\Sigma,\phi)})$ is non-trivial then so is $w_2(M_{(\wtd{\Sigma},\wtd{\phi})})$.

For the other implication recall that, on a $5$-manifold, $w_2$ is the obstruction to finding an orthogonal $4$-frame.
Since $M_{(\Sigma,\phi)}$ is simply connected there exists an embedded ball $\D^{5}\subset M_{(\Sigma,\phi)}$ containing $\S^1 \times \D^4$. The restriction of $TM$ to $\D^5$  is trivial. Hence over the boundary of $M\setminus \S^1 \times \D^4$ the frame is given by a map $\S^1 \times \S^3 \to V_4(\R^5)$ into the Stiefel manifold $V_4(\R^5)$. 
Observe that since $\S^1 \times \D^4$ is contractible (inside $\D^5$) this map is homotopic to the trivial map. \\
As such it can be extended over $\DD^2 \times \S^3$, proving that $M_{(\wtd{\Sigma},\wtd{\phi})}$ admits a global $4$-frame. We conclude that $w_2(M_{(\wtd{\Sigma},\wtd{\phi})}) = 0$ if and only if $w_2(M_{(\Sigma,\phi)}) =0 $. 
Lastly, the classification from Section \ref{sec:simply-conn-5folds} shows that the above operation is equivalent to taking the connected sum with $\S^2 \times \S^3$.
\end{proof}

\subsection{Proof of \Cref{thm:sympl_fol_5folds}}
\label{sec:proof_sympl_fol}

\paragraph*{Case $\rank(H_2(M;\ZZ))\geq2$:}

By Barden's classification (Theorem \ref{thm:w2}) there is a simply connected manifold $\wtd{M}$ such that $M = \wtd{M} \# \S^3 \times \S^2$.
More precisely, $\wtd{M}$ is the unique simply-connected $5-$manifold satisfying $H^2(\wtd{M},\ZZ) \oplus \ZZ = H^2(M,\ZZ)$, and $w_2(\wtd{M}) = w_2(M)$. 
Since $\wtd{M}$ admits a contact structure by \cite[Theorem 8]{Gei91}, it follows from  \Cref{thm:SFonconnectedsumwithspheres} that $M$ admits infinitely many symplectic foliations.
\hfill\qedsymbol

\paragraph*{Cases $\S^5$ and $\S^3 \times \S^2$:}
A leafwise symplectic structure on the Lawson foliation \cite{Law71} of $\S^5$ has already been found by Mitsumatsu \cite{Mit18}.
The existence of infinitely many symplectic foliations on $\S^3 \times \S^2$ follows from applying \Cref{thm:SFonconnectedsumwithspheres} to $\S^5$. 
Note that the resulting foliation is just the product of $\S^2$ with the Reeb foliation on $\S^3$.
\hfill\qedsymbol

For the next two cases we use the observation from \cite{KwovKoe16} that for the Brieskorn manifold
\[ \Sigma(a) = \Sigma(a_0,\dots,a_3) := \{z \in \C^4 \cap \S^7 \mid \sum_{j=1}^3 z_j^{a_j} = 0 \},\]
comes equipped with a contact form
\begin{equation}\label{eqn:ctct_form_brieskorn}
 \alpha := \frac{i}{2} \sum_{j=1}^3(z_j \d \overline{z}_j - \overline{z}_j \d z_j).
\end{equation}
Moreover, the projection $\pi: \C^4 \to \C$ given by $(z_0,\dots,z_3) \mapsto z_0$ induces an open book decomposition supporting $\alpha$, whose page is symplectomorphic to the Brieskorn variety
\[ V_\varepsilon(a_1,a_2,a_3) = \{ z \in \C^3 \mid \sum_{j=1}^{3} z_j^{a_j} = \varepsilon\}.\]
As usual we denote by $\alpha_B$ the restriction of $\alpha$ to the binding $B:= \Sigma(a) \cap \{z_0 = 0\} = \Sigma(a_1,a_2,a_3)$.

\paragraph*{Case $H_2(M;\ZZ)=\ZZ_p\oplus\ZZ_p$ for $p$ relatively prime to $3$.}

Consider $\Sigma(p,3,3,3)$ and observe it is simply connected, with $H_2(\Sigma(p,3,3,3))=\ZZ_p\oplus \ZZ_p$ and vanishing second Stiefel Whitney class (c.f.\ for instance \cite{KwovKoe16}).
In other words, $\Sigma(p,3,3,3)$ is nothing else than the manifold $M_p$ from \Cref{sec:simply-conn-5folds}.

The binding of the open book described above is $\Sigma(3,3,3) \subset \S^5$. The restriction of the Hopf fibration $h:S^5 \to \CP^2$ induces a fibration of $\Sigma(3,3,3)$ over the complex curve $S:= \{\sum_{j=1}^3 z_j^3\} \subset \CP^2$. 
According to the genus formula (see e.g.\ \cite[Page 53]{ACGHBook}), 
$S$ is diffeomorphic to a torus and we denote by $\theta_1,\theta_2 \in \Omega^1(\Sigma(3,3,3))$ the pullback of the two angular forms on $\TT^2$. 
Since the Reeb vector field $R_B$ of $\alpha_B$
is tangent to the fibers of the Hopf fibration, it follows that $\d \alpha_B$ is a multiple of $\theta_1 \wedge \theta_2$. 
Hence,
\[ \gamma :=\theta_2,\quad \eta:= \alpha_B \wedge \theta_1,\]
defines a cosymplectic structure on $B$ satisfying $\gamma \wedge \d \alpha_B =0$. 
Furthermore, by \Cref{lem:BrieskornRestriction} below
$[\eta]$ is in the image of the restriction map $\iota^*:H^2(V(3,3,3)) \to H^2(\Sigma(3,3,3))$. 
Then it follows from \Cref{prop:from_OBD_to_sympl_fol,lem:cosymplecticboundaryperturbation} that $\Sigma(p,3,3,3)$ admits a symplectic foliation.

\begin{lemma}\label{lem:BrieskornRestriction}
Let $\mathbf{a}=(a_1,a_2,a_3)$, and $A$ the least common multiple of the $a_i$'s.
Consider the $\S^1-$action on $\C^3$ defined by
\[
\lambda \cdot(z_1,z_2,z_3) = 
(\lambda^{A/a_1} z_1, \lambda^{A/a_2} z_2,\lambda^{A/a_3} z_3),
\] 
and assume that the basic cohomology group $H^2_b(\Sigma(\mathbf{a});\ZZ)$ is 
generated by the differential of the contact form on $\Sigma(\mathbf{a})$ as in \Cref{eqn:ctct_form_brieskorn}).
Then, the restriction map 
\[ \iota^*:H^2(V(\mathbf{a})) \to H^2(\Sigma(\mathbf{a})),\]
is surjective.
\end{lemma}
Notice that the contact form $\alpha$ on $\Sigma(\mathbf{a})$ given by  \Cref{eqn:ctct_form_brieskorn} is \emph{not} a basic form, so that $[\d \alpha]$ is \emph{a priori} not zero in $H^2_b(\Sigma(\mathbf{a}))$; here we make the additional assumption that it generates the whole basic cohomology group.

\begin{proof}
Consider the $5-$dimensional Brieskorn manifold $\Sigma(A,\mathbf{a})= \Sigma(A,a_1,a_2,a_3) \subset \S^7$. 
The restriction of the projection 
\[
\pi:\C^{4} \to \C, \quad (z_0,\dots,z_3) \mapsto z_0,
\]
defines an open book decomposition of $\Sigma(A,\mathbf{a})$, with page $V_\varepsilon(\mathbf{a})$ and binding $\Sigma(\mathbf{a})$. 
Furthermore, there is a weighted $\S^1$-action on $\Sigma(A,\mathbf{a})\subset \C^4$ given by
\[ 
\lambda \cdot (z_0,z_1,z_2,z_3) = (\lambda z_0, \lambda^{A/a_1}z_1,\lambda^{A/a_2}z_2, \lambda^{A/a_3}z_3).
\]
Notice that this action restricts to the one described in the statement on the binding $\Sigma(3,3,3)=\{z_0=0\}$.

Let $\tau(\Sigma(\mathbf{a}))$ be a tubular neighborhood of the binding invariant under the $\S^1$-action, and consider (part of) the Mayer-Vietoris sequence in $\S^1$-basic cohomology:
\[  
H_b^2(\Sigma(A,\mathbf{a}) \setminus \Sigma(\mathbf{a})) \oplus H_b^2(\tau(\Sigma(\mathbf{a})))  
\xrightarrow{i^* - j^*} 
H_b^2(\partial \tau(\Sigma(\mathbf{a}))) 
\to
H^3_b(\Sigma(A,\mathbf{a})).
\]
 Recall that, for any proper 
$G$-manifold $M$, we have $H_b^\bullet(M) = H^\bullet(M/G)$, where the latter denotes singular cohomology, (see e.g.\ \cite[Theorem 30.36]{Michor}), and that $G$-equivariant maps which are $G$-equivariantly homotopic induce the same map in basic cohomology (see e.g.\ \cite[Lemma 30.34]{Michor}). 
This implies that
\[H^3_b(\Sigma(A,\mathbf{a})) = H^3(\CP(\mathbf{w})) = 0,\]
where $\CP(\mathbf{w})$ denotes the weighted projective space with weight $\mathbf{w}=(1,A/a_1,A/a_2,A/a_3)$, and the last equality follows from \cite[Theorem 1]{Kawasaki}.
(Notice that $H^3(\CP(\mathbf{w}))$ denotes here the singular homology of $\CP(\mathbf{w})$ as a quotient topological space, and \emph{not} its orbifold (singular) cohomology as an orbifold.)  
Moreover, since each $\S^1-$orbit intersects the page in a single point, we also have that
\[ 
H^2_b(\Sigma(A,\mathbf{a}) \setminus \Sigma(\mathbf{a})) 
=
H^2(V_\varepsilon(\mathbf{a})),
\]
and that, because $\partial \tau(\Sigma(\mathbf{a}))$ is foliated by $\S^1-$orbits and by  copies of the binding given by its intersection with the pages of the open book,  
\[
H_b^2(\partial \tau(\Sigma(\mathbf{a}))) 
=
H^2(\Sigma(\mathbf{a})).
\]
Lastly, note that $\tau(\Sigma(\mathbf{a})$ retracts, in a $\S^1$-equivariant way, onto the binding $\Sigma(\mathbf{a})$. 
Summarizing, the relevant part of the Mayer-Vietoris sequence can be rewritten as:
\[ 
H^2(V_\varepsilon(\mathbf{a})) \oplus H^2_b(\Sigma(\mathbf{a})) 
\xrightarrow{i^*-j^*} H^2(\Sigma(\mathbf{a})) 
\to 
0.
\]

To see that $i^*$ is surjective it suffices to show that $j^*=0$. 
For this, consider on $\Sigma(A,\mathbf{a})$ the contact form as in \Cref{eqn:ctct_form_brieskorn}.
Observe that its Reeb vector field $R$ is the infinitesimal generator of the $\S^1$-action. 
As such, $\d \alpha$ defines a class in $H^2_b(\Sigma(1,\mathbf{a}))$. 
Moreover, since $\alpha$ is supported by the open book, the restriction $\d\alpha_B$ of $\d\alpha$ to the binding $B=\Sigma(\mathbf{a})$ is again of the form in \Cref{eqn:ctct_form_brieskorn}, hence $[\d \alpha_B] \in H^2_b(\Sigma(\mathbf{a}))$ is, by assumption, a generator.

Hence, in order to prove that 
\[
j^*\colon H^{2}_b(\Sigma(\mathbf{a})) =H^2_b(\tau \Sigma(\mathbf{a})) \to 
H^{2}_b(\partial\tau\Sigma(\mathbf{a})) = H^2(\Sigma(\mathbf{a}))
\]
is the trivial map, one has to prove that the generator $[\d\alpha_B]$ is sent to $0$.
But this simply follows from the fact that $j^*[\d\alpha_B]$, seen in $H^2(\Sigma(\mathbf{a}))$, is the differential of $[\alpha_B]\in\Omega^1(\Sigma(\mathbf{a}))$.
\end{proof}

\paragraph*{Case $H_2(M;\ZZ)=\ZZ_q\oplus\ZZ_q$ for $q$ relatively prime to $2$.}
This case is analogous to the one above.
Here, we look instead at the Brieskorn manifold 
\[
\Sigma(q,4,4,2)=\{(z_0,z_1,z_2,z_3)\in\CC^4\cap\SSS^{7} \, \vert \, z_0^q + z_1^4 + z_2^4 + z_3^2= 0 \},
\]
which is, as observed in \cite{KwovKoe16}, a simply-connected spin manifold with $H_2$ equal to $\ZZ_q\oplus \ZZ_q$. 
In other words, $\Sigma(q,4,4,2)$ is diffeomorphic to $M_q$ of \Cref{sec:simply-conn-5folds}.

The open book we consider is again induced by the ambient projection $\CC^4\to \CC$, $(z_0,z_1,z_2,z_3)\to z_0$. 
This time, however, the pages are naturally symplectomorphic to the Brieskorn variety
\[
V(4,4,2)=\{(z_1,z_2,z_3)\in\CC^3 \, \vert \, z_1^4 + z_2^4 + z_3^2 = \varepsilon \},
\]
with boundary $\Sigma(4,4,2)$, which naturally lives in $\SSS^5$. 
Now, there is a natural orbi-bundle projection $\SSS^5\to \CC P^2(1,1,2)$, $(z_1,z_2,z_3)\to [z_1,z_2,z_3]$, where $\CC P^2(1,1,2)$ is the weighted projective space given by the quotient of $\CC^3$ by the action $\lambda \cdot (z_1,z_2,z_3)=(\lambda z_1,\lambda z_2,\lambda^2 z_3)$ for every $\lambda\in\CC\setminus\{0\}$.
We are interested at the restriction of this orbi-bundle to $\Sigma(4,4,2)\subset \SSS^5$. 
The complex orbi-surface $S=\{z_1^4+z_2^4+z_3^2=0\}\subset \CC P(1,1,2)$ (which is well defined because the polynomial is invariant under the action of $\CC\setminus\{0\}$ on $\S^5$) is actually smooth, so that the restriction $\Sigma(4,4,2)\to S$ is a smooth fibration.
Moreover, according to the genus formula for surfaces in weighted projective spaces \cite[Theorem 5.3.7]{Hos16}, $S$ has genus $1$, and is hence diffeomorphic to a $2-$torus $\TT^2$.
The proof now can be concluded word by word as in the case above.
\hfill \qedsymbol

\section{Constructions of conformal symplectic foliations}
\label{sec:constructions_conf_sympl_fol}
We describe here the turbulization procedure, as well as the cobordisms and surgeries needed for the constructions of (conformal) symplectic foliations in the subsequent sections.

\subsection{Gluing and Turbulization}
\label{sec:turbulization_and_gluing}

Let $(\F,\eta,\omega)$ be a symplectic foliation on $M$ tangent to the boundary.
To glue two such manifolds, the boundaries must be isomorphic as conformal symplectic manifolds. Moreover, to ensure the resulting foliation is smooth, we also need to control the variation of $(\F,\eta,\omega)$ in the direction normal to $\partial M$.

Recall that the linear holonomy of $\F$ is encoded in a leafwise cohomology class. If $\F = \ker \gamma$ it can be computed as follows. The integrability condition of $\F$ implies
\[ \d \gamma = \mu \wedge \gamma,\]
for some $\mu \in \Omega^1(M)$. The restriction $\mu_\F := \mu|_\F$ is closed and the cohomology class $[\mu_\F] \in H^1(\F)$ depends only on $\F$. We refer to $\mu_\F$ as a \emph{holonomy form} of $\F$.

To encode the behaviour of $(\F,\eta,\omega)$ near the boundary fix two $1-$forms $\mu,\nu$ on $\partial M$
and a collar neighborhood $k:(-\varepsilon,0] \times \partial M \to M$. 
Using the collar neighborhood we define
\[ M_\infty := M \cup_{\partial M, k} [0,\infty) \times \partial M,\]
and extend $(\F,\eta,\omega)$ over $[0,\infty) \times \partial M$ by:
\[ \F := \ker\, (-t \mu + \d t),\quad \eta := \eta_\partial + \nu,\quad \omega:= \omega_\partial.\]
If this extension is smooth we say that the collar is \emph{$(\mu,\nu)-$adapted}.

\begin{definition}
A conformal symplectic foliation is \emph{$(\mu,\nu)$-tame at the boundary}, if it admits a $(\mu,\nu)-$adapted collar neighborhood as above.
If $\mu = \pm \nu$, we will say \emph{positive/negative $\mu$-tame}, and if $\mu = \nu = 0$, simply \emph{tame at the boundary}.
\end{definition}

Although the above definition makes sense for any choice of $\mu$ and $\nu$, is it most significant when the Lee form $\eta$ of the leafwise conformal symplectic structure coincides (at least up to sign) with the holonomy form $\mu$. 
In fact, this condition makes sense globally, and not only around the boundary.

\begin{definition}\label{def:hol_like}
A conformal symplectic foliation $(\F,\eta,\omega)$ is said to be \holonomy if 
$\eta$ equals a holonomy form $\mu_\F$ of $\F$. If the equality only holds up to sign we call the foliation \absholonomy.
\end{definition}
Note that in the definition of \absholonomy we do not ask the sign to be the same at every point. That is, in the region where $\mu_\F = 0$, the sign can change.

As stated above, the tameness condition ensures that when gluing manifolds along their boundaries the resulting conformal symplectic foliation is smooth. The precise statement is as follows: 
\begin{proposition}
\label{thm:gluing_turbul}
Let $(M_i,\F_i,\eta_i,\omega_i)$, $i=1,2$ be (conformal) symplectic foliations $\mu_i-$tame at the boundary. 
If there exists an orientation reversing diffeomorphism $\phi:\partial M_1\to \partial M_2$ respecting the induced (conformal) symplectic structures on the boundary
and the holonomy forms $\mu_i$, 
then $M_1\cup_\phi M_2$ admits a (conformal) symplectic foliation. 
\end{proposition}

As in the symplectic setting (\Cref{sec:turbulization_and_gluing}) we use turbulization to produce conformal symplectic foliations that are tame at the boundary. 
Unlike symplectic structures, (twisted-)exact conformal symplectic structures also exist on closed manifolds; consider for example the conformal symplectization of a contact manifold. 
Thus it is no surprise that in addition to boundaries of cosymplectic type, conformal turbulization also applies to boundaries of contact type, where these notions are defined as follows:

\begin{definition}\label{def:FoliatedConformalBoundaryTypes}
A conformal symplectic foliation $(M,\F,\eta,\omega)$ transverse to the boundary $\partial M$ is said to be of:
\begin{itemize}
\item \emph{cosymplectic type} if there exist an admissible form $\alpha \in \Omega^1(\F_\partial)$ for $\omega_\partial$ satisfying $\d_{\eta_\partial} \alpha = 0$;
\item  convex/concave \emph{contact type} if there exists an admissible form $\alpha \in \Omega^1(\F_\partial)$ for $\omega_\partial$ satisfying $\d_{\eta_\partial} \alpha =  \pm\omega_\partial$ (c.f.\ \Cref{lem:contacttypeboundaryequivalence}).
\end{itemize}
\end{definition}

In both cases we can use turbulization to change the foliation to become tangent to the boundary. Since the proofs are slightly different in each case the constructions are stated in separate theorems.

\begin{theorem}
\label{thm:ConformalTurbulization}
Let $(M,\F,\eta,\omega)$ a conformal symplectic foliation with contact type boundary, and such that $\F_\partial=\ker\gamma_\partial$ is unimodular. 
Then, there exists a conformal symplectic foliation $(\wtd{\F},\wtd{\eta},\wtd{\omega})$ on $M$ such that:
\begin{enumerate}[(i)]

    \item
    If $(\F,\eta,\omega)$ has \emph{convex contact boundary}, the boundary leaf of $(\wtd{F},\wtd{\eta},\wtd{\omega})$ is
    \[ (\partial M,\wtd{\eta}_\partial = \eta_\partial \pm \gamma_\partial, \wtd{\omega}_\partial = \d_{\eta_\partial \pm \gamma_\partial} \alpha),\]
    for any $\alpha \in \Omega^1(\partial M)$ satisfying $\omega_\partial = \d_{\eta}\alpha$. 
    Moreover, $(\wtd{\F},\wtd{\eta},\wtd{\omega})$ is either of the following:
    \begin{enumerate}
        \item tame at the boundary;
        \item positive/negative $\gamma_\partial$-tame at the boundary, where the sign corresponds to the sign in the equation above. 
        If $\eta_\partial = 0$ then the resulting foliation is also holonomy-like.
    \end{enumerate}
    
    \item
    If $(\F,\eta,\omega)$ has \emph{concave contact boundary}, the boundary leaf of $(\wtd{F},\wtd{\eta},\wtd{\omega})$ is
   \[ (\wtd{\eta}_\partial = \eta_\partial \mp \gamma_\partial, \wtd{\omega}_\partial = \d_{\eta_\partial \mp \gamma_\partial} \alpha),\]
    for any $\alpha \in \Omega^1(\partial M)$ satisfying $\omega_\partial = \d_{\eta}\alpha$. Moreover,  $(\wtd{\F},\wtd{\eta},\wtd{\omega})$ is either of the following:
    \begin{enumerate}
        \item tame at the boundary;
        \item positive/negative $(-\gamma_\partial)$-tame  at the boundary, where the sign corresponds to the sign in the equation above. If $\eta_\partial = 0$ then the resulting foliation is |holonomy|-like.
    \end{enumerate}
\end{enumerate}
    Moreover, in either case above, on an arbitrary small neighborhood of the boundary we have $\wtd{\omega} = \d_{\wtd{\eta}} \wtd{\lambda}$ with $\wtd{\lambda}_\partial = \alpha$, and $(\wtd{\F},\wtd{\eta},\wtd{\omega})$ agrees with $(\F,\eta,\omega)$ away from this neighborhood. 
    Furthermore, if $\omega = d_\eta \lambda$ globally on $(M,\F)$, then $\wtd{\omega}=\d_{\wtd{\eta}}\wtd{\lambda}$ globally, with $\wtd{\lambda}$ extending $\lambda$.
\end{theorem}

\begin{theorem}
\label{thm:ConformalTurbulization_cosymplectic}
Let $(M,\F,\eta,\omega)$ be a conformal symplectic foliation with boundary of cosymplectic type, and such that $\F_\partial = \ker \gamma_\partial$ is unimodular.
Then, there exists a conformal symplectic foliation $(\wtd{\F},\wtd{\eta},\wtd{\omega})$ on $M$ tame at the boundary and such that:
\begin{enumerate}[(i)]
    \item\label{item:thm_conformal_turbul_1} $(\F,\eta,\omega)$ and $(\wtd{F},\wtd{\eta},\wtd{\omega})$ agree away from an arbitrarily small neighborhood of $\partial M$.
    \item The conformal symplectic structure on the boundary leaf is given by
    \[ \wtd{\omega}_\partial = \omega_\partial 
    \pm
    \alpha \wedge \gamma_\partial,\]
    where $\alpha \in \Omega^1(\F_\partial)$ is any admissible form satisfying $\d_{\eta_\partial} \alpha = 0$,
    and the sign can be chosen arbitrarily.
\end{enumerate}
\end{theorem}

Even though we will not directly need it in the rest of the paper, it is interesting to point out that by adding a further interval-worth of compact leaves near the boundary one can always obtain a conformal symplectization as conformal symplectic structure at the boundary leaf:
\begin{corollary}
\label{cor:conformal_turbulization}
Let $(M,\F,\eta,\omega)$ be a conformal symplectic foliation with $\F_\partial$ unimodular and of (convex or concave) contact type, with $\eta_\partial$ admitting a closed extension to $\partial M$. 
Then, there exists an conformal symplectic foliation $(\wtd{\F},\wtd{\eta},\wtd{\omega})$ on $M$, tame at the boundary, such that the following holds.
\begin{enumerate}
    \item
    $(\F,\eta,\omega)$ and $(\wtd{\F},\wtd{\eta},\wtd{\omega})$ agree away from an arbitrarily small neighborhood of $\partial M$.
    \item
    The conformal symplectic structure on the boundary leaf is given by
    \[ \wtd{\omega}_\partial = \d_{\pm \gamma_\partial} \alpha,\]
    where the sign can be chosen,
    $\gamma_\partial \in \Omega^1(\partial M)$ is the closed form defining $\F_\partial$ and $\alpha \in \Omega^1(\partial M)$ is any form satisfying $\omega_\partial = \d_{\eta}\alpha$.
    \item $\wtd{\F}$ restricts to the foliation by closed leaves $\{pt\}\times\partial M$ on a sufficiently small neighborhood of the boundary $(-\delta,0]\times \partial M$.
\end{enumerate}
\end{corollary}

We now give the proof of \Cref{thm:ConformalTurbulization} and \Cref{cor:conformal_turbulization} above.

\begin{proof}[Proof of \Cref{thm:ConformalTurbulization}]
We start by proving the statement in the case that $(\F,\eta,\omega)$ has convex boundary of contact type.
Using \Cref{thm:FoliatedConformalSymplecticNormalform} and the contact type boundary condition we fix a collar neighborhood $(-\varepsilon,0] \times \partial M$ on which
\[ \F = \ker \gamma_\partial, \quad \eta = \eta_\partial,\quad \omega = \d_{\eta_{\partial}}(e^t\alpha),\]
where $\gamma_\partial$ is a closed $1-$form defining $\F_\partial$ (which exists by assumption of unimodularity), and, with an abuse of notation, $\eta_\partial$ is seen as a closed $1-$form defined on $\partial M$ (which extends the given leaf-wise form by assumption).
Furthermore, we choose smooth functions $f,g:(-\varepsilon,0] \to [0,1]$ satisfying $(f,g)\neq (0,0)$ everywhere and
\[ 
f(t) = \begin{cases} 1 & \text{for $t$ near $-\varepsilon$}  \\ -t & \text{for $t$ near $0$} \end{cases},\quad g(t)  = \begin{cases} 0 & \text{for $t$ near $-\varepsilon$} \\ 1 & \text{for $t$ near $0$} \end{cases},
\]
with $\dot{f}\leq 0$ and $\dot{g}\geq 0$.

Then we define
\begin{equation}\label{eq:TurbulizationDefiningForms}
    \wtd{\gamma} = f(t)\gamma_\partial 
    \pm
    g(t)\d t,\quad \wtd{\eta} :=  \eta_\partial
    \mp \frac{\dot{f}}{g}
    \gamma_\partial,\quad \wtd{\omega} = \d_{\eta_\partial
    \mp\frac{\dot{f}}{g}\gamma_\partial}( e^{h(t)}\alpha),
\end{equation} 
where $h\colon(-\epsilon,0]\to\RR$ equals $e^t$ near $t=-\epsilon$, has $\dot{h}(t)>0$ for $t<0$, and has all derivatives zero at $t=0$.
Notice that $\wtd{\eta},\wtd{\omega}$ are well defined everywhere on $\wtd{\F}$, since
\[
\wtd{\gamma}\wedge \wtd{\eta} = \wtd{\gamma} \wedge \eta_\partial -\dot{f}\d t \wedge \gamma_\partial.
\]

To see that $(\wtd{\F},\wtd{\eta},\wtd{\omega})$ is a conformal symplectic foliation observe that $\wtd{\eta}$ and $\wtd{\omega}$ are, respectively, $\d$- and $\d_{\wtd{\eta}}$-closed on $\wtd{\F}$. The leafwise non-degeneracy is checked as follows:

\begin{align*}
    \wtd{\gamma} \wedge \wtd{\omega}^n &= (f \gamma_\partial  
    \pm
    g \d t) \wedge (\d_{\eta_\partial}(e^{h(t)}\alpha) 
    \pm e^{h}\frac{\dot{f}}{g}
    \gamma_\partial \wedge \alpha)^n \\
    &= f \gamma_\partial \wedge \d_{\eta_\partial}(e^h\alpha)^n 
    \pm g  dt \wedge d_{\eta_\partial}(e^h\alpha)^n +n\dot{f}
    e^h \d t \wedge \gamma_\partial \wedge \alpha \wedge \d_{\eta_\partial}(e^h \alpha)^{n-1} \\
    &=
    n f \dot{h} e^{nh} \gamma_\partial \wedge \d t \wedge \alpha \wedge \d\alpha^{n-1}
    \pm g 
    e^{nh} \d t \wedge (\d_{\eta_\partial}\alpha)^n
    +ne^{n h}\dot{f}
    \d t \wedge \gamma_\partial \wedge \alpha \wedge \d \alpha^{n-1} 
    \\
    &\overset{(*)}{=} 
    n e^{nh}
    (
    f\dot{h} -\dot{f}
    )
    \, \gamma_\partial \wedge \d t \wedge \alpha \wedge \d\alpha^{n-1} > 0,
\end{align*}
as desired.
Here, for $(*)$ we used that $\d_{\eta_\partial}\alpha^n = \omega_\partial^n = 0$.

The foliation $\wtd{\F} := \ker \wtd{\gamma}$ is positive/negative $\gamma_\partial$-tame at the boundary and agrees with $\F$ away from the boundary.
Moreover, the restriction of $(\wtd{\eta},\wtd{\omega})$ to $\F$ agrees with $(\eta,\omega)$ near $t = -\varepsilon$.

Lastly, the statement about exactness of $\wtd{\omega}$ under the assumption of exactness of $\omega$, as well as the fact that $\wtd{\eta}$ is a holonomy form for $\wtd{\gamma}$ in the case where $\eta_\partial=0$, are clear for the explicit formulas above.
Moreover, $\F$ can be arranged to be tame at the boundary by choosing $f$ constant on a neighborhood of $t=1$. In this case we define the leafwise conformal symplectic structure by:
\[
\wtd{\eta} \coloneqq \eta_\partial \pm \gamma_\partial,
\quad
\wtd{\omega} \coloneqq \d_{\eta_\partial \pm \gamma_\partial}(e^t\alpha) .
\]
Observe that $\eta$ is not a holonomy form for $\wtd{\gamma}$ in this case.

If instead $(\F,\eta,\omega)$ has concave boundary of contact type, then Theorem \ref{thm:FoliatedConformalSymplecticNormalform} provides a collar neighborhood $(-\varepsilon,0] \times \partial M$ on which
\[ \F = \ker \gamma_\partial,\quad \eta = \eta_\partial, \quad \omega = \d_{\eta_\partial} (e^{-t}\alpha).\]
Here $\alpha \in \Omega^1(\partial M)$ satisfies $\d_{\eta_\partial} \alpha = \omega$ and $\gamma_\partial \wedge \alpha \wedge \omega_\partial^{n-1} < 0$. 
The same computations as above then show that defining

\[ \wtd{\gamma} := f(t) \gamma_\partial \pm g(t) \d t,\quad \wtd{\eta}:= \eta_\partial \pm \frac{\dot{f}}{g} \gamma_\partial,\quad \wtd{\omega} := \d_{\wtd{\eta}}(e^{h(t)}\alpha),\]
gives the desired conformal symplectic foliation. 
Moreover, the explicit formula shows that the conformal symplectic leafwise structure is exact if $\omega$ is, with primitive $\lambda$ as in the statement, and that $\eta$ minus a holonomy form. Also, same considerations as in the convex case hold for arranging $\F$ tame at the boundary.
\end{proof}

\begin{proof}[Proof of \Cref{thm:ConformalTurbulization_cosymplectic}]
For the case that $(\F,\eta,\omega)$ has cosymplectic boundary, we use again Theorem \ref{thm:FoliatedConformalSymplecticNormalform} to find a collar neighborhood $(-\varepsilon,0] \times \partial M$ on which
\[ \F = \ker \gamma_\partial,\quad \eta = \eta_\partial, \quad \omega = \omega_\partial + \d_{\eta_\partial}(t \alpha),\]
for an admissible form $\alpha$ satisfying $\d_{\eta_\partial} \alpha = 0$. Then a straightforward computation shows that
\[ 
\wtd{\gamma} = f \gamma_\partial \pm g \d t ,
\quad
\wtd{\eta} := \eta_\partial ,
\quad 
\wtd{\omega} = \omega_\partial + \d t \wedge \alpha \pm \alpha\wedge \gamma_\partial,
\]
defines the desired conformal symplectic foliation.
\end{proof}

\begin{proof}[Proof of \Cref{cor:conformal_turbulization}]
Consider the trivial cobordism $[0,1] \times \partial M$ endowed with the conformal symplectic foliation
\[ \left( \F := \bigcup_{t \in [0,1]} \{t\} \times \partial M,\quad \eta := f(t) \eta_\partial + \gamma_\partial,\quad \omega := \d_{\eta} \alpha\right),\]
where $f:[0,1]\to \R$ is a smooth function satisfying
\[ 
f(t) = 
\begin{cases} 
1 & \text{for $t$ near $0$} \\ 
0 & \text{for $t$ near $1$.} 
\end{cases}
\]

\noindent
First applying Theorem \ref{thm:ConformalTurbulization} and then gluing the above cobordism to the boundary completes the proof.
\end{proof}

\subsection{Conformal symplectization of homotopic contact structures.}
\label{sec:conf_sympl_homot_contact_str}

\begin{lemma}
\label{prop:interpolating_conformal_symplectizations_homotopic_contact_str}

Let $\alpha_0$ and $\alpha_1$ be homotopic (through almost contact forms) contact forms on a manifold $N$ of dimension at least $5$. Then $\S^1 \times [0,1] \times N$ admits a conformal symplectic foliation tame at the boundary such that:
\begin{enumerate}[(i)]
\item The boundary leaves are isomorphic to $(\S^1 \times N, \d_{\pm \d \theta} \alpha_i)$, $i=0,1$, where the signs can be chosen freely.
\item The foliation contains at most one closed leaf in the interior isomorphic to $(\S^1 \times N, \d_{\pm \R \d\theta} \alpha_{ot}$, where the sign and $R> 0$ can be chosen freely, and $\alpha_{ot}$ is any overtwisted contact form in the almost contact class of $\alpha_i$. If one of the $\alpha_i$ is overtwisted we may arrange the foliation to have no closed interior leaves.
\end{enumerate}
\end{lemma}

\begin{proof}
Suppose first that $\alpha_0$ and $\alpha_1$ are tight.
According to \cite[Theorem 1.1]{EliMur15} there is a Liouville structure on $[0,1] \times N$ which restricts to $\alpha_{ot}$ and $\alpha_{0}$ on the concave and convex boundary respectively. 
Thus, the product $ \S^1 \times[0,1] \times N$ admits a symplectic foliation which can be turbulized using \Cref{thm:ConformalTurbulization} (with the choice of $\gamma_\partial$ equal to $\pm R\d\theta$ and $\pm \d\theta$ at the two ends). 
The resulting boundary leaves equal 
\[ (\S^1 \times N , \d_{\pm R\d \theta} \alpha_{ot}),\quad \text{and} \quad (\S^1 \times N,\d_{\pm \d \theta} \alpha_0).\]
Repeating the argument for $\alpha_1$ and gluing the resulting manifolds using \Cref{thm:sympl_gluing_turbul} proves the claim.

In the case where $\alpha_0$ is overtwisted, one can just apply the first part of the argument above with $\alpha_0$ instead of $\alpha_{ot}$.
If $\alpha_1$ is overtwisted, one can apply the above proof to $\alpha_0'=\alpha_1$ and $\alpha_1'=\alpha_0$, then apply the orientation preserving diffeomorphism
\[
\S^1\times[0,1]\times M \to \S^1\times[0,1]\times M, \quad 
(\theta,t,p)\mapsto (-\theta,-t,p)
\]
in order to conclude.
\end{proof}

\subsection{Round connected sum}
\label{sec:fol_round_handle}

The classification of simply connected $5-$manifolds (\Cref{sec:simply-conn-5folds})
is up to connected sum. 
In light of this, we would like a connected sum construction for (conformal) symplectic foliations;
however, such a construction is not available even in the setting of smoothly foliated manifolds. 
This being said, there is a notion of $\S^1-$connected sums in the smoothly foliated setting, which results in a smooth foliation on the $\S^1-$connected sum of two foliated manifolds along two positively transverse curves, in such a way that the resulting leaves are connected sums of the leaves of the original two foliations.

As the connected sum of two symplectic manifolds is not symplectic in general
(cf.\ \cite[Lemma 7.2.2]{McDSalBook}),
one cannot expect to be able to do the same with symplectically foliated manifolds. 
To avoid this problem we slightly modify the construction using turbulization. To be precise, we work with \emph{$\S^1$-connected sums} of manifolds defined as follows.
Let $\gamma_i \subset M_i^m$, $i=1,2$ be (embedded) closed curves and identify their tubular neighborhoods with $\S^1 \times \D^{m-1}$. Then we define the \emph{$\S^1$-connected sum} as:
\[ M_1 \#_{\S^1} M_2 := \left(M_1 \setminus \Int\S^1 \times \D^{m-1}) \right) \cup_\partial \left(M_2 \setminus \Int (\S^1 \times \D^{m-1})\right),\]
where the boundaries are glued by identifying
\[ (\theta,x) \sim  (-\theta,x),\quad \theta \in \S^1, x \in \partial \D^{m-1}.\]

\begin{proposition}\label{prop:round_sum_conf_sympl_fol}
Let $\gamma_i \subset M_i$ be a closed curve (positively) transverse to a conformal symplectic foliation $(\F_i,\eta_i,\omega_i)$, for $i = 1,2$. Then $M_1 \#_{\S^1} M_2$ admits a conformal symplectic foliation.
\end{proposition}
\begin{proof}
A tubular neighborhood of each of the curves can be identified with $\S^1 \times \D^{2n}$, where $\dim M_i = 2n+1$. Furthermore, using a leafwise Moser trick we may assume that on these neighborhoods the conformal symplectic foliation is equivalent to
\[ \F = \bigcup_{ \theta \in \S^1} \{ \theta \} \times \D^{2n},\quad \eta = 0, \quad \omega = \omega_{st}.\]
Thus after removing (a slightly smaller) $\S^1 \times \D^{2n}$ the foliations are transverse to the boundary and of contact type (Defintion \ref{def:FoliatedConformalBoundaryTypes}).
As such they can be turbulized using \Cref{thm:ConformalTurbulization}. The conformal symplectic structure on the resulting boundary leaves $\S^1 \times \S^{2n-1}$ are
\[ \eta_\partial = \pm \d \theta,\quad \omega_\partial = \d_{\pm \d \theta} \alpha_{st}.\]
Hence, if we choose opposite signs on each of the pieces we can glue using  \Cref{thm:gluing_turbul}, giving a conformal symplectic foliation on $M_1 \#_{\S^1} M_2$.
\end{proof}

It turns out that connected sums of simply connected $5$-folds can also be obtained by $\S^1$-connected sums. 
This will be key in Section \ref{sec:proof_conf_sympl_fol_5folds}.

\begin{lemma}
\label{lemma:round_conn_sum_homology}
Consider two $5$-dimensional manifolds $M_1$ and $M_2$ such that $M_1$ is simply connected and contains a closed curve $\gamma_1$, while $\pi_1(M_2) = \ZZ$ and generated by a closed curve $\gamma_2$.
Then the associated $\S^1$-connected sum satisfies:
\begin{enumerate}
    \item $H_2(M_1 \#_{\S^1} M_2) = H_2(M_1) \oplus H_2(M_2)$;
    \item $\pi_1(M_1 \#_{\S^1} M_2) = 1$;
    \item $w_2(M_1 \#_{\S^1} M_2) = 0$ if and only if $w_2(M_1) = 0$ and $w_2(M_2) = 0$.
\end{enumerate}
\end{lemma}

\begin{proof}
For $i=1,2$, denote $M_i^\circ := M_i\setminus \gamma_i$.
Then, applying Mayer--Vietoris we find
\begin{equation*}
 \to H_2(\SSS^1\times\SSS^{3}) \to 
H_2(M_1^\circ)\oplus H_2(M_2^\circ) \to
H_2(M_1 \#_{\S^1} M_2) \to 
H_1(\SSS^1\times\SSS^{3}) \to
H_1(M_1^\circ)\oplus H_1(M_2^\circ)\to,
\end{equation*}
which reduces to 
\begin{equation}
\label{eqn:mayer_vietoris_round_handle_to_conn_sum}
0 \to 
H_2(M_1^\circ)\oplus H_2(M_2^\circ)\to
H_2(M_1 \#_{\S^1} M_2) \to 
\ZZ\xrightarrow{h}
H_1(M_1^\circ)\oplus H_1(M_2^\circ).
\end{equation}
Furthermore, we claim that $H_1(M_i^\circ)=H_1(M_i)$ and $H_2(M_i^\circ)=H_2(M_i)$, $i=1,2$.
Indeed, again applying Mayer--Vietoris gives:
\begin{multline*}
H_2(\SSS^1\times\SSS^{3})\to H_2(M_i^\circ)\oplus H_2(\SSS^1\times\DD^{4})\to
H_2(M_i) \to H_1(\SSS^1\times \SSS^{3})\xrightarrow{f}
\\ 
H_1(M_i^\circ)\oplus H_1(\SSS^1\times \DD^4) \to H_1(M_i) \to H_0(\SSS^1\times\SSS^{3})\xrightarrow{g} H_0(M_i^\circ)\oplus H_0(\SSS^1\times\DD^{4}).
\end{multline*}
As the maps $f$ and $g$ are injective, this implies: 
\begin{equation*}
    0\to H_2(M_i^\circ)\to
H_2(M_i) \to 0 ,
\quad 
 0\to  H_1(\SSS^1\times \SSS^{3}) \to H_1(M_i^\circ)\oplus H_1(\SSS^1\times \DD^{4}) \to H_1(M_i) \to 0.
\end{equation*}
In particular, $H_1(M_i^\circ)=H_1(M_i)$ and $H_2(M_i^\circ)=H_2(M_i)$ as desired.
Going back to \Cref{eqn:mayer_vietoris_round_handle_to_conn_sum}, since $\gamma$ is a generator in homology $h$ is injective. Thus we find $H_2(M_1 \#_{\S^1} M_2) = H_2(M_1) \oplus H_2(M_2)$.

That $M_1 \#_{\S^1} M_2$ is simply connected follows immediately from Van Kampen's theorem using the same covering as above, so we omit the details. Hence it remains to show the relation for the second Stiefel-Whitney class.

Consider the following portion of the Mayer-Vietoris sequence:
\begin{multline}
\label{eqn:proof_w2}
H^1(M_1 \#_{\S^1} M_2)\to
H^1(M_1^\circ)\oplus H^1(M_2^\circ) \to
H^1(\S^1\times \S^3)\to\\
H^2(M_1 \#_{\S^1} M_2)\to
H^2(M_1^\circ)\oplus H^2(M_2^\circ)\to 
H^2(\S^1\times \S^3)
\end{multline}

Here, $H^2(\S^1\times \S^3)=0$, as well as $H^1(M)=H^1(M_1^\circ)=0$, because both $M_1 \#_{\S^1} M_2$ and $M_1^\circ$ are simply-connected. Moreover, the restriction map $H^1(M_2^\circ) \to H^1(\S^1 \times \S^3)$ is surjective.
As such, \Cref{eqn:proof_w2} reduces to
\[
0\to
H^2(M) \xrightarrow{\sim}
H^2(M_1^\circ) \oplus H^2(M_1^\circ) \to 0.
\]
In particular, by naturality of the Stiefel-Whitney classes, $w_2(M)=w_2(M_1^\circ) \oplus w_2(M_2^\circ)$, so that $w_2(M)$ is trivial if and only if both $w_2(M_1^\circ)$ and $w_2(M_2^\circ)$ are.

Hence, it is left to prove that
$w_2(M_i)$ is trivial if and only if $w_2(M_i^\circ)$ is.
This can be argued using the interpretation of $w_2$ as the obstruction to find a trivialization of the tangent bundle over the $2-$skeleton of a cellular decomposition of the manifold.
Start by choosing a cellular decomposition of $\S^1\times\S^3$ without $2-$cells, and extend it to a cellular decomposition of $\S^1\times \DD^4$ also without $2-$cells; this can be further extended to a cellular decomposition of $M_i$.
Hence, by construction the induced cellular decomposition on $M_i^\circ=M_i\setminus \Int(\S^1\times\DD^4)$ has the same $2-$skeleton as that of $M_1$.
In particular, as $TM_i^\circ = TM_i\vert_{M_1^\circ}$, the obstruction to find a trivialization of the tangent bundle of $M_i$ and $M_i^\circ$ over their respective (coinciding) $2-$skeleton is the same; in other words, $w_2(M_i)$ is trivial if and only if $w_2(M_i^\circ)$ is, as desired.
\end{proof}

\subsection{Connected sums with exotic spheres}
\label{sec:conf_sympl_fol_conn_sum_exotic_spheres}

Let $f:\S^{n-1} \to \S^{n-1}$ be an (orientation preserving) diffeomorphism. The \emph{twisted sphere} with twist $f$ is defined by
\[ S_f := \D^n \cup_f \D^n.\]
The diffeomorphism type of $S_f$ depends only on the isotopy class of $f$. Thus we may assume that $f$ is a compactly supported diffeomorphism of $\R^{2n+1}$, extended as the identity to $\S^{2n+1}$.

Milnor \cite{Mil59b} showed that suitable choices of $f$ yield exotic spheres. 
As a consequence, Lawson \cite[Corollary 6]{Law71} showed that exotic spheres admit (smooth) foliations. 
To be precise, consider an embedding of $\R^{2n}$ into a leaf of a foliation on $\S^{2n+1}$. Cutting $\S^{2n+1}$ along this region and gluing back using a diffeomorphism $f$ we obtain $S_f$. Since the cut is made inside a leaf the foliation is preserved.

In turn, we use the results in \cite[Appendix by S. Courte]{CKS18} to adapt this statement to conformal symplectic foliations. That is, we prove that if $M$ admits a conformal symplectic foliation then so does the connected sum with any exotic sphere:

\begin{proof}[Proof of \Cref{thm:conf_sympl_fol_up_to_exotic_sphere}]
Suppose $M$ admits a conformal symplectic foliation $(\F,\eta,\omega)$, and consider a loop $\gamma \subset M$ transverse to $\F$. On a neighborhood $\S^1 \times \D^{2n}$ of $\gamma$ the conformal symplectic foliation in equivalent to
\[ \F = \bigcup_{\theta \in \S^1 } \{\theta\} \times \D^{2n},\quad \eta = 0 ,\quad \omega = \omega_{st}.\]

Cut $M$ along the cyclinder $\S^1 \times \S^{2n-1}$ contained in the above local model, and turbulize both boundaries using \Cref{thm:ConformalTurbulization}. 
Then the resulting conformal symplectic foliation has two boundary leaves each isomorphic to $ \left(\S^1 \times \S^{2n-1},\, \eta_\partial = \d \theta,\, \omega_\partial = \d_{\d\theta} \alpha_{st}\right)$.
We insert the trivial cobordism $[0,1]\times \S^1 \times \S^{2n-1}$ with the conformal symplectic foliation from \Cref{prop:interpolating_conformal_symplectizations_homotopic_contact_str}. 
This smoothly recovers $M$ but with a different conformal symplectic foliation, which contains a compact leaf
\[ \left(\S^1 \times \S^{2n-1}, \eta = R \d\theta,\, \omega = \d_{R \d\theta} \alpha_{ot}\right),\]
for $R> 0$ arbitrarily large. 
Now, if $R\gg R_0>0$, the piece of symplectization $\left( ( -R_0,R_0) \times \S^{2n-1}, \d(e^t \alpha_{ot}) \right)$ can be embedded in this compact leaf.

Consider now the symplectization $\R^{2n}$ of the standard overtwisted contact structure on $\R^{2n-1}$. 
By \cite[Theorem 9.2]{CKS18} any given compactly supported diffeomorphism of $\R^{2n}$ can be realized (up to isotopy) by a compactly supported symplectomorphism of the symplectization.
By extending as the identity, these symplectomorphisms can be realized in $\left((-R_0,R_0)\times \S^{2n-1}, \d(e^t\alpha_{ot})\right)$ provided $R_0>0$ is big enough (which we hence assume).

To conclude the proof, one can then simply cut $M$ along $(-R_0,R_0) \times \S^{2n-1}$, and glue it back via the compactly supported symplectomorphism in the smooth (compactly supported) isotopy classes of exotic (compactly supported) diffeomorphisms of $(-R_0,R_0) \times \S^{2n-1}$ given as in the previous paragraph. 
As the starting exotic diffeomorphism varies, the resulting conformally foliated manifold can then be made diffeomorphic to the connected sum of $M$ with any given exotic sphere $\Sigma$.
\end{proof}

\section{Conformal symplectic foliations on simply-connected $5$-manifolds}
\label{sec:proof_conf_sympl_fol_5folds}

We prove here \Cref{thm:conf_sympl_fol_5folds}, stating the existence of  conformal symplectic foliations on every simply-connected $5-$manifold.
The strategy we follow goes along the lines of that in \cite{ACa72}, although with adaptations in order to ensure that the constructed foliations are conformal symplectic.

First, as a direct consequence of \Cref{thm:ConformalTurbulization,thm:gluing_turbul} we obtain the following construction:
\begin{corollary}
\label{prop:double_conf_sympl_fol}
Let $(W_i,\omega_i)$, $i=1,2$ be symplectic manifolds with the same (convex/concave) contact boundary $B$, and $\phi_i:W_i\to W_i$ a symplectomorphism which is the identity near the boundary. Then the gluing 
\[ (W_1)_{\phi_1} \cup_{B \times \SSS^1} (W_2)_{\phi_2},\]
of the mapping tori $(W_i)_{\phi_i}$
admits a conformal symplectic foliation.
\end{corollary}
The foliation is obtained by turbulizing the natural symplectic foliation on each of the mapping tori, and then gluing along their common boundary.

We use this result to construct symplectically foliated mapping tori $N_k$ satisfying $H_2(N_k)=\ZZ_k\oplus \ZZ_k$.
Then, we describe explicit conformal symplectic foliations on $\S^5$, $\S^2\times \S^3$, $X_\infty$ and $X_\infty \# \S^2\times \S^3$. 
Lastly, in order to get a conformal symplectic foliation on any simply connected $5-$manifold, we take round-connect sum (\Cref{sec:fol_round_handle}) of properly chosen $N_k$'s, together with a symplectically foliated $\SSS^5$ or $X_\infty$ (c.f.\ \Cref{sec:simply-conn-5folds} for the definition of $X_\infty$) and as many copies of symplectically foliated $\SSS^2\times\SSS^3$ as needed to get the right free part of $H_2$.

\paragraph*{Constructing the mapping tori with $H_2 = \ZZ_k\oplus \ZZ_k$.}
Let $k\geq1$, and 
consider the Brieskorn manifold $S=\Sigma(2,2,k,1)=\{z_0^2+z_1^2+z_2^k+z_3=0\} \cap \SSS^7\subset \CC^4$. Notice that $S$ is naturally diffeomorphic to $\SSS^5$. Consider moreover the contact open book decomposition
\[ S\to \CC , (z_0,z_1,z_2,z_3) \mapsto z_3 .
\]
This open book decomposition has pages which are naturally diffeomorphic to the Brieskorn variety $V\coloneqq V(2,2,k)=\{z_0^2+z_1^2+z_3^k=\epsilon\}\subset \CC^3$. 
Denote then the symplectic monodromy of this open book by $\phi\colon V \to V$,
and its mapping torus by $V_\phi$.
Consider also the double $N_k$ of $V_\phi$, i.e.\ two copies of $V_\phi$ (one with the orientation reversed) glued along their boundary via the identity map.
As done for instance in \cite[Section 4]{DurLaw72}, one can explicitly compute that $H_2(E)=\ZZ_k\oplus \ZZ_k$. 

Lastly, notice that,  by \Cref{prop:double_conf_sympl_fol}, $N_k$ admits a conformal symplectic foliation which just coincides with the fibers of the mapping tori away from the gluing region in the construction of the double.

\paragraph*{Constructing the mapping tori with $H_2 = \ZZ\oplus \ZZ$.}
Consider this time the link of singularity $S=\{(z_0+z_1^2)(z_0^2+z_1^5)+z_2^2+z_3=0\}\cap\S^7\subset \CC^4$.
Again, $S$ is naturally diffeomorphic to $\S^5$.
Moreover, the projection 
\[ S\to \CC , \quad  (z_0,z_1,z_2,z_3) \mapsto z_3 
\]
describes an open book decomposition of $\S^5$ with pages which are naturally diffeomorphic to $T\coloneqq \{(z_0+z_1^2)(z_0^2+z_1^5)+z_2^2=\epsilon\}\subset \CC^3$. 
Let $\psi\colon T \to T$ be the symplectic monodromy,
and $T_\psi$ its mapping torus.
Let also $P$ be the double of $T_\psi$, constructed analogously to the $N_k$'s above.
As explained again in \cite[Section 4]{DurLaw72}, one can explicitly compute that $H_2(E)=\ZZ\oplus \ZZ$. 
What's more,  by \Cref{prop:double_conf_sympl_fol}, $P$ also admits a conformal symplectic foliation, coinciding with the fibers of the mapping tori away from the gluing region of the double.

\paragraph*{Conformal symplectic foliations on $\SSS^5$, $\SSS^2\times\SSS^3$ and $X_\infty \# \S^2\times \S^3$.}
These manifolds actually admit symplectic foliations.
The case of $\SSS^5$ is described in \cite{Mit18} (and then put in the general framework for instance in \cite[Chapter 6]{TorresThesis}).
On $\SSS^2\times\SSS^3$ one can just use the symplectic foliation  defined by the product of the Reeb foliation on $\SSS^3$ by the $\SSS^2-$factor.
Lastly, on $X_\infty \# \S^2\times\S^3$ we just use the symplectic foliations constructed in \Cref{thm:sympl_fol_5folds}. 

\paragraph*{Conformal symplectic foliation on $X_\infty$.}
Recall from the classification in Section \ref{sec:simply-conn-5folds} that $X_\infty = S^2 \wtd{\times} \S^3$ is the nontrivial $\S^3$-bundle over $S^2$. Such bundles can be constructed by choosing a loop of diffeomorphisms $\phi: \S^1 \to \Diff(\S^3)$ and using it to glue two trivial $\S^3-$bundles over $\DD^2$ along their common boundary $\S^1\times\S^3$:
\[ \D^2 \times \S^3 \cup_\phi \D^2 \times \S^3.\]
Up to isomorphism the resulting bundle only depends on $[\phi] \in \pi_1(\Diff(\S^3))$; moreover, according to \cite{Hat83}, the latter group is isomorphic $ \ZZ/2$.
In particular, if $[\phi]$ is non-trivial in $\pi_1(\Diff(\S^3))$, the gluing above gives 
$X_\infty$.

Consider now a loop in $\S^5$, transverse to the symplectic foliation  constructed by Mitsumatsu \cite{Mit18}. 
On a neighborhood $\S^1\times \DD^4$ of this loop the symplectic foliation looks like
\[ \left(\S^1 \times \D^4, \F = \bigcup_{z \in \S^1} \{z\} \times \D^4,\omega = \omega_{st}\right),\]
where $\omega_{st}$ is the standard symplectic structure on $\D^4$.

Removing this neighborhood yields a symplectic foliation on 
$\S^5\setminus \S^1\times \DD^4$
which on a neighborhood $(-\varepsilon,0] \times \S^3 \times \S^1$ looks like
\begin{equation} 
\label{eqn:sympl_fol_Xinfty}
\F = \ker \d\theta,\quad \omega = \d(e^{-t}\alpha_3),
\end{equation}
where $t\in (-\varepsilon,0]$, $ \theta \in \S^1$, and $\alpha_3 \in \Omega^1(\S^3)$ denotes the standard contact structure.
Notice that, as any two loops in $\S^5$ are isotopic, then one simply has $\S^5\setminus \S^1\times \DD^4 = \DD^2\times \S^3$ (because this is true for the standard unknot).
In other words, we get a symplectic foliation on $\DD^2\times\S^3$ which on a neighborhood $(-\varepsilon,0] \times \S^3 \times \S^1$ of the boundary looks like in \Cref{eqn:sympl_fol_Xinfty}.

Now, applying the turbulization construction from \Cref{thm:ConformalTurbulization} 
we obtain a conformal symplectic foliation tame at the boundary, and restricting to
\begin{equation}\label{eq:ConformalSymplecticFoliationXinfty}
    \left(\{0\}\times\S^3 \times \S^1,\, \eta = \d \theta,\, \omega = \d_{\d \theta} \alpha_3\right).
\end{equation}

If we consider $\S^3 \subset \C^2$ then a generator of $\pi_1(\Diff(\S^3))$ is given by the path
\[ \phi_\theta(z_1,z_2) = (e^{i\theta}z_1, z_2),\]
for $\theta\in\S^1$; see for example \cite[p.263]{Pras06}. 
Moreover it is easily checked that $\phi_t$ preserves the standard contact form on $\S^3$. 
This implies that the gluing map 
\[ \phi:\S^1 \times \S^3 \to \S^1 \times \S^3,\quad (\theta,x) \mapsto (\theta,\phi_\theta(x)),\]
used to construct $X_\infty$, is an isomorphism of the conformal symplectic structure from Equation \ref{eq:ConformalSymplecticFoliationXinfty}. 
Hence, we obtain a conformal symplectic foliation on $X_\infty$.

\paragraph*{Proof of \Cref{thm:conf_sympl_fol_5folds}.}
Let now $M$ be the almost contact simply-connected $5-$manifold on which we need to construct the conformal symplectic foliation.
According to Smale-Barden's statement, it is enough to construct a conformal symplectic foliation on an almost contact, simply-connected manifold $M'$ such that it has the same $H_2$ as $M$ and $w_2(M')$ is trivial if and only if $w_2(M)$ is.

Let's start by dealing with the case of $M$ spin, i.e.\ $w_2(M)=0$.
Write then $H_2(M)=\ZZ^r \oplus \bigoplus_j \ZZ_{k_j}\oplus \ZZ_{k_j}$.
We now further distinguish two cases: $r=2k$ even or $r=2k+1$ for some $k\geq0$.
\begin{itemize}
\item[$\bullet$] In the case of $r=2k$, we consider the round connected sum $M'$ of the following conformal-symplectically foliated spin manifolds: 
$\SSS^5$ with the Mitsumatsu foliation and a transverse curve $\gamma_0$ in it;
then, for each $j$, the mapping tori $N_{k_j}$ with curves $\gamma_{k_j}$ transverse to the foliations and such that they generate the (non-torsion) fundamental group $\pi_1(N_{k_j})$;
and lastly $k$ copies of the conformally symplectically foliated mapping torus $T$ with $k$ disjoint transverse curves $\delta_1,\ldots,\delta_k$.
\item[$\bullet$] In the case of $r=2k+1$, we consider as $M'$ the round connected sum  of the following conformal-symplectically foliated spin manifolds: 
$\S^2\times\S^3$ with the above described symplectic foliation, and a transverse curve $\gamma_0$ in it; then,
for each $j$, the mapping tori $N_{k_j}$ with curves $\gamma_{k_j}$ transverse to the foliations and such that they generate the (non-torsion) fundamental group $\pi_1(N_{k_j})$;
and lastly $k$ disjoint copies of the conformally symplectically foliated mapping torus $T$ with $r$ disjoint transverse curves $\delta_1,\ldots,\delta_r$.
\end{itemize}
(Notice that all these transverse curves exist by construction of the foliations.)
Using \Cref{prop:round_sum_conf_sympl_fol} 
and \Cref{lemma:round_conn_sum_homology} 
inductively (notice that in each set of generators for the connect sum, there is only one which is simply connected, while in the others the chosen curves are non-torsion and generate the $\pi_1$), $M'$ can be seen to be spin, satisfy $H_2(M')=H_2(M)$, and to admit a conformal symplectic foliation, as desired.
\\

Let's now deal with the non-spin case.
Write here $H_2(M)=\ZZ^{r+1} \oplus \bigoplus_j \ZZ_{k_j}\oplus \ZZ_{k_j}$ (notice that in the non-spin almost contact case, $\rank(H_2(M;\ZZ))\geq1$ by the Smale-Barden's classification).
As before, we further distinguish two cases: $r=2k$ or $r=2k+1$ for some $k\geq0$.
\begin{itemize}
\item[$\bullet$] In the case of $r=2k$, we consider the round connected sum $M'$ of the following conformal-symplectically foliated spin manifolds: 
$X_\infty$ with the conformal symplectic foliation described above and a transverse curve $\gamma_\infty$ in it;
then, for each $j$, the mapping tori $N_{k_j}$ with curves $\gamma_{k_j}$ transverse to the foliations and such that they generate the (non-torsion) fundamental group $\pi_1(N_{k_j})$;
and lastly $k$ copies of the conformally symplectically mapping torus $T$ with $k$ disjoint transverse curves $\delta_1,\ldots,\delta_k$.
\item[$\bullet$] In the case of $r=2k+1$, we consider as $M'$ the round connected sum  of the following conformal-symplectically foliated spin manifolds: 
$X_\infty\# \S^2\times \S^3$ with the above described conformal symplectic foliation, and a transverse curve $\gamma_\infty'$ in it; 
then, for each $j$, the mapping tori $N_{k_j}$ with curves $\gamma_{k_j}$ transverse to the foliations and such that they generate the (non-torsion) fundamental group $\pi_1(N_{k_j})$;
and lastly $k$ disjoint copies of the conformally symplectically foliated mapping torus $T$ with $r$ disjoint transverse curves $\delta_1,\ldots,\delta_r$.
\end{itemize}
Using again \Cref{prop:round_sum_conf_sympl_fol} 
and \Cref{lemma:round_conn_sum_homology}
inductively, $M'$ can be seen to have non-trivial $w_2$, to satisfy $H_2(M')=H_2(M)$, and to admit a conformal symplectic foliation.

\section{Existence of conformal symplectic foliations  in high dimensions}
\label{sec:general_existence}

\subsection{Conformal symplectic foliated cobordisms}
\label{sec:conf_sympl_fol_cobordisms}

Consider an oriented manifold $M^{2n}$ endowed with a \emph{contact foliation} $(\F,\xi :=\ker \alpha)$. That is, $\F$ is a (cooriented, codimension-$1$) foliation and $\alpha \in \Omega^1(\F)$ satisfies
\[ \alpha \wedge \d \alpha^{n-1} \neq 0.\]
Since $\F$ is cooriented its leaves inherit an orientation from $M$. We say a contact foliation $(\F,\xi = \ker \alpha)$ is \emph{positive} (resp. \emph{negative}) if $\alpha$ is a positive (resp. negative) contact form on $\F$.

Recall that $\lambda \in \Omega^1(M)$ is  a \emph{Liouville form} for the conformal symplectic structure $(\eta,\omega)$ if 
\[ \omega = \d_\eta \lambda.\]
Such a form is equivalent to the choice of a \emph{Liouville vector field} $Z \in \X(M)$, uniquely defined by
$ \lambda = \iota_Z \omega$.
A cooriented hypersurface $\Sigma \subset (M,\eta,\omega)$ is of \emph{positive} (resp. \emph{negative}) contact type, if there is a Liouville form $\lambda$ for $\omega$ whose restriction defines a positive (resp. negative) contact form on $\Sigma$. This is equivalent to the Liouville vector field being positively (resp. negatively) transverse to $\Sigma$, with respect to its coorientation.

\begin{definition}
We say that $(W,\F,\eta,\omega)$ is a \emph{foliated conformal symplectic cobordism} between contact foliated manifolds $(M_\pm,\F_\pm, \xi_\pm)$ if:
\begin{enumerate}
\item $W$ is an oriented cobordism from $M_-$ to $M_+$, i.e. $\partial W = \overline{M}_- \sqcup M_+$.
\item $(\F,\eta,\omega)$ is a conformal symplectic foliation transverse to the boundary which admits a (leafwise) Liouville form $\lambda$ near the boundary such that:
\begin{enumerate}
    \item $\F \cap M_\pm = \F_\pm$;
    \item $\lambda_\pm := \lambda|_{M_\pm}$ is a (positive) contact form defining $\xi_\pm$.
\end{enumerate}
\end{enumerate}
We will also denote such a cobordism simply by an arrow:
\[ (M_-,\F_-,\xi_-) \to (M_+,\F_+,\xi_+).\]
\end{definition}

\begin{remark}
The second condition above is equivalent to the existence of a (leafwise) Liouville vector field $Z$ for $(\eta,\omega)$ which is transverse to $\partial W$. We denote by $\partial_+ W$ (resp. $\partial_- W$) the part of $\partial W$ where $Z$ is pointing outwards (resp. inwards). Then, the above definition implies:
\[ \partial_+ W = M_+, \quad \partial_- W = \overline{M}_-,\]
and $\lambda_\pm = \iota_Z\omega|_{\partial_\pm W}$, defines (on $\partial_\pm W$) a positive/negative contact form for $\xi_\pm$.
Also note that the equality in the first condition is as cooriented foliations. Hence comparing $\partial_-W = \overline{M}_-$ to $M_-$ it is the orientation of the leaves of $\F_-$ which changes. This is consistent with $\lambda_-$ being a (leafwise) negative contact form on $\partial_-W$ and a (leafwise) positive contact form on $M_-$.
\end{remark}

Composition of morphisms is just composition of cobordisms (using \Cref{thm:FoliatedConformalSymplecticNormalform}) and the identity morphism at $(M,\F,\alpha)$ is the symplectization
\[ \left([0,1]\times M, \omega := \d(e^t \alpha)\right).\] 

Recall that a (contact) foliation is called unimodular if it can be defined by a closed $1$-form. 
In this setting,
we can use turbulization to ``invert'' cobordisms.

\begin{proposition}\label{prop:CobordismEquivalenceRelation}
The foliated conformal symplectic cobordism relation defines an equivalence relation, between manifolds with a unimodular contact foliation.
\end{proposition}

\begin{proof}
We can compose cobordisms using \Cref{thm:FoliatedConformalSymplecticNormalform},
and a cobordism from $(M,\F = \ker \gamma, \xi = \ker \alpha)$ to itself is given by the symplectization
\[ \left([0,1] \times M , \G = \ker \gamma, \omega = \d(e^t\alpha)\right).\]
 Hence the cobordism relation is transitive and reflexive. 

To see it is symmetric consider a cobordism $(W,\F := \ker \gamma, \eta, \omega)$ from $(M_-,\F_- := \ker \gamma_-, \xi_- := \ker \alpha_-)$ to $(M_+,\F_+ := \ker \gamma_+, \xi_+ := \ker \alpha_+)$.
As a smooth manifold, the cobordism in the opposite direction is given by
\[ \wtd{W} := [-\varepsilon,0] \times M_+ \cup_{M_+} \overline{W} \cup_{M_-} [0,\varepsilon] \times M_-,\]
where the gluing uses the (orientation reversing) identity map. Observe that
\[ \partial \wtd{W} = \overline{M}_+ \sqcup M_-,\]
so that $\wtd{W}$ is an oriented cobordism from $M_+$ to $M_-$. 

On the middle piece $\overline{W}$ consider the conformal symplectic foliation $(\F = \ker - \gamma, \eta, \omega)$. That is, we swap the coorientation of $\F$ to match changing the orientation on $W$. By turbulizing at both boundaries (\cref{thm:ConformalTurbulization}) we obtain a foliation which is tame at the boundary and has compact leaves:
\[ \left(M_+,\eta_+ + \gamma_+, \d_{\eta_+ + \gamma_\partial} \alpha_+\right),\quad \text{and} \quad \left(M_-,\eta_-+\gamma_-,\d_{\eta_- + \gamma_\partial} \alpha_-\right),\]
where $\eta_\pm := \eta|_{M_\pm}$ and $\gamma_\pm \coloneqq \gamma\vert_{M_\pm}$.
Note that the coorientation of the turbulized foliation is pointing inwards (into $\overline{W}$) along $M_+$ and outwards along $M_-$.
On the left piece, $[-\varepsilon,0] \times M_+$, consider the conformal symplectic foliation
\[ \left(\F := \ker \gamma_+,\eta_+,\d_{\eta_+}(e^t \alpha_+)\right).\]
Again, using \Cref{thm:ConformalTurbulization}, we turbulize this foliation at $\{0\} \times M_+$ to make it tame with boundary leaf
\[ \left(M_+, \eta_++\gamma_+, \d_{\eta_+ + \gamma_+} \alpha_+\right).\]
The coorientation of the compact leaf is pointing outwards.
This matches the negative boundary of the middle piece, so that the resulting (cooriented) foliation is smooth. The resulting left boundary is the contact foliated manifold
\[ \left(\overline{M}_+, \wtd{\F}_- = \ker \gamma_+,\wtd{\xi}_- = \ker \alpha_+\right)\]
Similarly we obtain a conformal symplectic foliation on the right piece $[0,\varepsilon] \times M_-$, which glues smoothly to the middle piece and whose positive boundary is the contact foliated manifold
\[ \left( M_- = \partial_- W,\wtd{\F}_+ = \ker \gamma_-,\wtd{\xi}_+ = \ker \alpha_-\right).
\qedhere
\]
\end{proof}

The proof of \Cref{prop:CobordismEquivalenceRelation} also shows that by attaching trivial cobordisms and turbulizing, we can transfer (connected components of) the positive boundary to the negative boundary and vice versa. More formally:

\begin{lemma}\label{lem:CobordismBoundarySwapping}
Given two contact foliations $(M_i,\F_i,\xi_i)$, $i=1,2$, the following are equivalent:
\begin{enumerate}
\item $(M_0,\F_0,\xi_0) \to (M_1,\F_1,\xi_1)$;
\item $(M_0,\F_0, \xi_0) \sqcup (\overline{M}_1,\overline{\F}_1,\xi_1) \to \emptyset$;
\item $\emptyset \to (\overline{M}_0,\overline{\F}_0,\xi_0) \sqcup (M_1,\F_1,\xi_1)$.
\end{enumerate}
Moreover, the corresponding cobordisms can be taken diffeomorphic.
\end{lemma}

\begin{proof}
Suppose $(W,\F,\eta,\omega)$ is a cobordism as in the first statement. Then a neighborhood of the positive boundary is isomorphic to
\[ \left((-\varepsilon,0] \times M_1, \F = \ker \gamma_1,\eta = \eta_1,\omega = \d_{\eta_1} (e^t \alpha_1)\right),\]
where $\F_1 = \ker \gamma_1$, and $\xi_1 = \ker \alpha_1$. By turbulizing we can make it tame at the boundary such that the boundary leaf is cooriented outwards and equal to $\left(M_1,\d_{\eta_1 + \gamma_1}\alpha_1\right)$.
Similarly, the trivial cobordism
\[ \left([0,1] \times M_1, \F := - \ker \gamma_1, \eta:= \eta_1,\omega := \d_{\eta_1}(e^{-t} \alpha_1)\right),\]
can be turbulized at $\{0\} \times M_1$. The resulting boundary leaf equals $(M_1,\d_{\eta_1 + \gamma_1} \alpha_1)$ as is cooriented inwards (i.e. $\partial_t$ is positively transverse to it). As such the two cobordisms can be glued. 

The Liouville vector field of the resulting cobordism points inwards at all boundary components, showing that
\[ (M_0,\F_0,\xi_0) \sqcup (\overline{M}_1, \overline{\F}_1,\xi_1) \sim \emptyset.\]
The other implication follows similarly.
\end{proof}

\subsection{From almost symplectic to conformal symplectic}
\label{sec:from_alm_sympl_to_conf_sympl}

For simplicity, we introduce the following notation for product foliations.
Given a cooriented positive contact structure $\xi$ on an oriented smooth manifold $M$, we denote by $\S^1\times(M,\xi)$ the contact foliation $(\F=\ker(\d\theta),\xi)$. 
The product foliation of $\SSS^1$ with of an (almost) symplectic manifold is defined analogously.
In particular, \Cref{thm:EliMurphyMain} from \cite{EliMur15} tells that product of $\SSS^1$ with almost symplectic foliated cobordisms from overtwisted to tight contact foliations can be homotoped to symplectic ones.

It is worth pointing out that in in the (non-foliated) symplectic setting the assumption that the negative boundary is overtwisted is necessary, i.e.\ \Cref{thm:EliMurphyMain} is in general false without this assumption.
This is simply a consequence of the fact that overtwisted contact manifolds are not symplectically fillable \cite{BEM15,Nie06}. 
In particular, tight and overtwisted structures play a very different role with respect to the symplectic cobordism relation. 
In the foliated setting this difference vanishes, as we shall see with some results in this section.

\begin{remark} 
\label{rmk:hol_like_EM}
Throughout this whole section, we will often use either \Cref{thm:EliMurphyMain} and \Cref{rmk:weak_relat_to_relat}, or \Cref{thm:MakingCobordismSymplectic}, on product almost symplectic foliations $(\SSS^1\times W, \F=\ker\d\theta,\Omega)$, with $\theta\in\SSS^1$.
Notice that these results easily give a holonomy-like exact conformal symplectic leafwise structure on $\F$.
Indeed, in their proofs one can easily see that the $(\eta,\omega)$ obtained satisfies that $\omega=\d_\eta\lambda$, for an exact $\eta$ which is $0$ near the boundary.
Then, $\eta$ can be realized as a holonomy form of a $1-$form $\gamma$ defining $\F$ which is just (a constant multiple of) $\d\theta$ near the boundary.
\end{remark}

Recall that \Cref{prop:CobordismEquivalenceRelation} says that the conformal symplectic foliated cobordism relation is symmetric.
In what follows, we will need a more precise version of \Cref{prop:CobordismEquivalenceRelation} under the assumption of (topologically) trivial cobordisms, which allows to keep track of the underlying almost contact structure.

\begin{lemma}
\label{lemma:cobord_between_homotopic_str_transverse_boundary}
Consider a homotopy of almost contact structures $(\alpha_t,\omega_t)$ between contact forms $\alpha_0$ and $\alpha_1$ on $M$.
Suppose that $\alpha_1$ is overtwisted, and if $\dim M = 3$ additionally assume $\alpha_0$ is overtwisted.
Then, there exists a $\vert$holonomy$\vert$-like exact conformal symplectic foliated cobordism
\[ \left(\S^1 \times [0,1] \times M, \F
,\eta,
\d_\eta\lambda): \S^1 \times (M,\alpha_0) \to \S^1 \times (M,\alpha_1).
\right),\]
Moreover:
\begin{enumerate}[(i)]
    \item $(\F,\eta,\d_\eta\lambda)$ is in the same almost contact class (relative boundary) as $ \left(\ker \d \theta, \d t \wedge \alpha_t + \omega_t\right)$;
    \item  Near the left (resp. right) boundary $\lambda$ equals $e^t\alpha_0$ (resp.\ $e^t\alpha_1$) and $\eta=0$.
    \item $(\eta,\d_\eta\lambda)$ is holonomy-like 
    away from a closed leaf.
\end{enumerate}
\end{lemma}

\begin{proof}
Starting from the (trivial) almost symplectic foliation
\[ \left(\S^1 \times [0,1] \times M, \F_0 := \ker d\theta, \omega_0 := \d t \wedge \alpha_t + \omega_t,\right)\]
the homotopy to $(\F,\eta,\omega)$ is described in two steps. 
First we homotope the hyperplane distribution to the desired foliation, keeping track of the almost symplectic structure.
This homotopy is depicted in \Cref{fig:foliations}. The resulting foliation contains two closed leaves in the interior, diffeomorphic to $\S^1 \times M$.
After this, we keep the foliation fixed and homotope the leafwise almost symplectic structure into a conformal symplectic structure.

First, to define the homotopy of almost symplectic foliations, consider the following $1$-parameter families: 
\[ \gamma_s := \cos(\varphi_s(t)) \d \theta - \sin(\varphi_s(t)) \d t,
\quad 
\Omega_s :=  \left(\cos(\varphi_s(t)) \d t + \sin(\varphi_s(t)\right) \d \theta) \wedge \alpha_t + \omega_t,
\]
where $s \in [0,1]$, and  $\varphi_s$ is a $[0,1]-$family of functions as in \Cref{fig:homotopy_foliations}. 
(For later use, we also assume that $\varphi_1$ is such that $|\dot{\varphi}_1(t) \sin(\varphi_1(t))|$ is a smooth function, which can be arranged as the picture shows.)
It is easily seen that $(\gamma_s,\Omega_s)$ defines a homotopy of almost symplectic foliations since:
\[ 
\gamma_s \wedge \d \gamma_s = 0,\quad \gamma_s \wedge \Omega_s^n = n  \d \theta \wedge \d t \wedge \alpha_t \wedge \omega_t > 0 
.\]
Secondly, we want to homotope $\Omega_1$ to a leafwise exact conformal symplectic structure on $\F_1$ of the form
\[ \Omega_2 := \d_{\eta} \lambda,\]
where $\eta$ is defined by
\[ \eta := |\dot{\varphi} \sin(\varphi)| \, (\cos(\varphi) \d t + \sin(\varphi) \d \theta), \]
with $\varphi\coloneqq \varphi_1(t)$, and $\lambda$ satisfies
\begin{equation}\label{eq:LambdaConditions}
    \lambda = \begin{cases} e^{g(t)} \alpha_0 
    &
    \text{on the left component of }\supp(\pi-\varphi)
    \\ 
    e^{g(t)} \alpha_1 
    & 
    \text{on the right component of }\supp(\pi-\varphi),
    \end{cases}
\end{equation}
for $g:[0,1] \to \R$ is as in \Cref{fig:function_g}. Note that up to sign $\eta$ equals the holonomy form of $\gamma_1$, so that $\Omega_2$ is \absholonomy. The conditions in \Cref{eq:LambdaConditions} implies that around the compact leaves $\Omega_2$ equals the result of the turbulization construction \Cref{thm:ConformalTurbulization}.

It remains to find suitable $\lambda$, and show that the resulting $\Omega_2$ is homotopic to $\Omega_{1,1}$ through leafwise almost symplectic forms.
To this end observe that a homotopy $(\alpha_{t,s},\omega_{t,s})$ of $1$-parameter families of almost contact structures on $M$ induces a homotopy of leafwise almost symplectic forms
\[ \Omega_{1,s} := \left( \cos(\varphi) \d t + \sin (\varphi) \d \theta\right) \wedge \alpha_{t,s} + \omega_{t,s}\]
We apply this with a homotopy $(\alpha_{t,s},\omega_{t,s})_{s\in[0,1]}$ satisfying the following properties:
\begin{enumerate}[(i)]
    \item $(\alpha_{t,s}, \omega_{t,s}) = (\alpha_t,\omega_t)$ for $(s,t) \in \{0\} \times [0,1] \cup [0,1] \times \{0,1\}$. 
    This ensures the homotopy starts at $(\alpha_t,\omega_t)$ and is relative to the boundary;
    
    \item The homotopy ends at
    \[ 
    (\alpha_{t,1},\omega_{t,1}) = 
    \begin{cases} 
    \left( e^{g(t)} \alpha_0, e^{g(t)} \d \alpha_0\right) 
    & 
    \text{on the left component of }\supp(\pi-\varphi)
    \\ 
    \left( e^{g(t)} \alpha_1, e^{g(t)}  \d \alpha_1\right) 
    & 
    \text{on the right component of }\supp(\pi-\varphi), 
    \end{cases}\]
    with $g$ as before.
\end{enumerate}

On $\supp(\pi-\varphi)$ we then define $\lambda$ by \Cref{eq:LambdaConditions}. A straightforward computation shows that, in this region, the linear interpolation from $\Omega_{1,1}$ to $\Omega_2$ is through leafwise almost symplectic forms on $\F_1$, and is relative to the boundary.
Lastly, on the middle piece $\SSS^1\times[0,1]\times M \setminus \supp(\pi - \varphi)$ one can obtain the desired extension of $\lambda$, and a homotopy (relative to the boundary) from $\Omega_{1,1}$ to $\d_\eta \lambda$ simply by using \cite{EliMur15} and \Cref{rmk:weak_relat_to_relat}.
\end{proof}

\begin{figure}[t]
     \centering
     \begin{subfigure}[b]{0.6\textwidth}
         \centering
         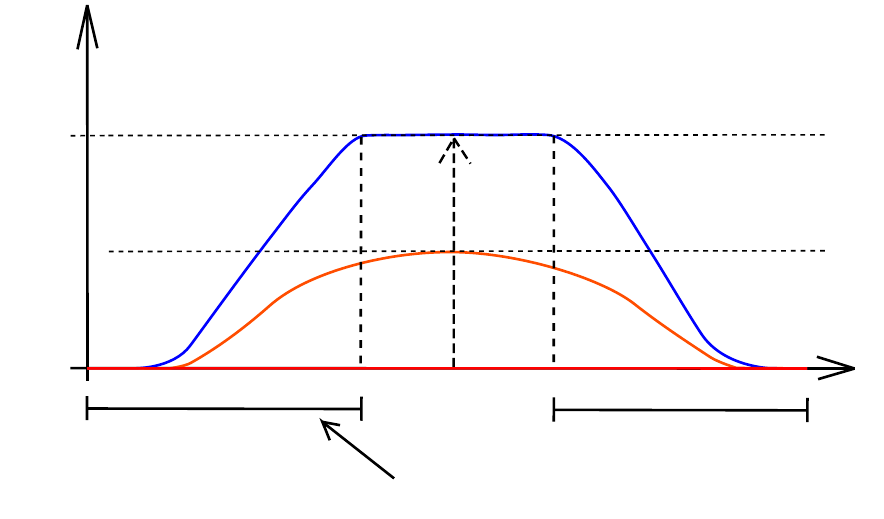
         \caption{Graphs of the phase function $\phi_s$ for different values of $s$. }
         \label{fig:homotopy}
     \end{subfigure}
     \hfill
     \begin{subfigure}[b]{0.35\textwidth}
         \centering
\begingroup%
  \makeatletter%
  \providecommand\color[2][]{%
    \errmessage{(Inkscape) Color is used for the text in Inkscape, but the package 'color.sty' is not loaded}%
    \renewcommand\color[2][]{}%
  }%
  \providecommand\transparent[1]{%
    \errmessage{(Inkscape) Transparency is used (non-zero) for the text in Inkscape, but the package 'transparent.sty' is not loaded}%
    \renewcommand\transparent[1]{}%
  }%
  \providecommand\rotatebox[2]{#2}%
  \newcommand*\fsize{\dimexpr\f@size pt\relax}%
  \newcommand*\lineheight[1]{\fontsize{\fsize}{#1\fsize}\selectfont}%
  \ifx\svgwidth\undefined%
    \setlength{\unitlength}{162.59678628bp}%
    \ifx\svgscale\undefined%
      \relax%
    \else%
      \setlength{\unitlength}{\unitlength * \real{\svgscale}}%
    \fi%
  \else%
    \setlength{\unitlength}{\svgwidth}%
  \fi%
  \global\let\svgwidth\undefined%
  \global\let\svgscale\undefined%
  \makeatother%
  \begin{picture}(1,0.69479598)%
    \lineheight{1}%
    \setlength\tabcolsep{0pt}%
    \put(0,0){\includegraphics[width=\unitlength,page=1]{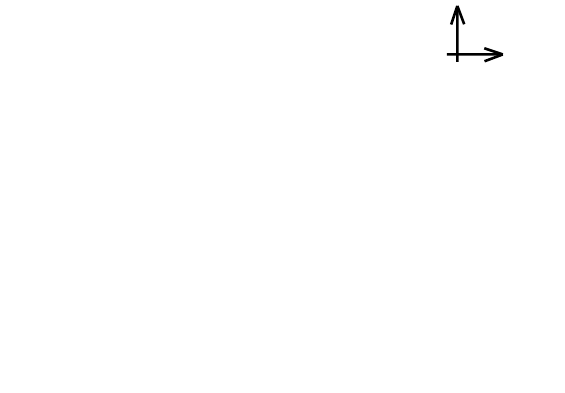}}%
    \put(0.74339258,0.64806911){\makebox(0,0)[lt]{\lineheight{1.25}\smash{\begin{tabular}[t]{l}$\theta$\end{tabular}}}}%
    \put(0.86161677,0.54204294){\makebox(0,0)[lt]{\lineheight{1.25}\smash{\begin{tabular}[t]{l}$t$\end{tabular}}}}%
    \put(0,0){\includegraphics[width=\unitlength,page=2]{foliations.pdf}}%
    \put(0.06490796,0.61672964){\makebox(0,0)[lt]{\lineheight{1.25}\smash{\begin{tabular}[t]{l}$s=0$\end{tabular}}}}%
    \put(-0.00510512,0.38130361){\makebox(0,0)[lt]{\lineheight{1.25}\smash{\begin{tabular}[t]{l}$s=1/2$\end{tabular}}}}%
    \put(0.07381874,0.06859062){\makebox(0,0)[lt]{\lineheight{1.25}\smash{\begin{tabular}[t]{l}$s=1$\end{tabular}}}}%
  \end{picture}%
\endgroup%

         \caption{The resulting foliations.}
         \label{fig:foliations}
     \end{subfigure}
        \caption{The initial homotopy to an almost symplectic foliation in the proof of \Cref{lemma:cobord_between_homotopic_str_transverse_boundary}.}
        \label{fig:homotopy_foliations}
\end{figure}

\begin{figure}[t]
     \centering
     \def\svgwidth{0.35\textwidth}
\begingroup%
  \makeatletter%
  \providecommand\color[2][]{%
    \errmessage{(Inkscape) Color is used for the text in Inkscape, but the package 'color.sty' is not loaded}%
    \renewcommand\color[2][]{}%
  }%
  \providecommand\transparent[1]{%
    \errmessage{(Inkscape) Transparency is used (non-zero) for the text in Inkscape, but the package 'transparent.sty' is not loaded}%
    \renewcommand\transparent[1]{}%
  }%
  \providecommand\rotatebox[2]{#2}%
  \newcommand*\fsize{\dimexpr\f@size pt\relax}%
  \newcommand*\lineheight[1]{\fontsize{\fsize}{#1\fsize}\selectfont}%
  \ifx\svgwidth\undefined%
    \setlength{\unitlength}{126.17708788bp}%
    \ifx\svgscale\undefined%
      \relax%
    \else%
      \setlength{\unitlength}{\unitlength * \real{\svgscale}}%
    \fi%
  \else%
    \setlength{\unitlength}{\svgwidth}%
  \fi%
  \global\let\svgwidth\undefined%
  \global\let\svgscale\undefined%
  \makeatother%
  \begin{picture}(1,0.65234608)%
    \lineheight{1}%
    \setlength\tabcolsep{0pt}%
    \put(0,0){\includegraphics[width=\unitlength,page=1]{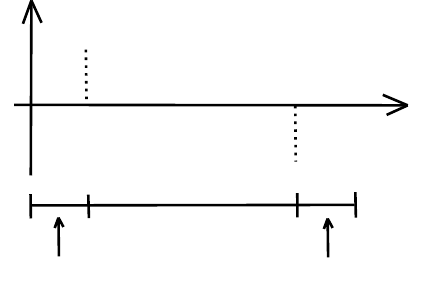}}%
    \put(0.04215668,0.01113655){\makebox(0,0)[lt]{\lineheight{1.25}\smash{\begin{tabular}[t]{l}slope $+1$\end{tabular}}}}%
    \put(0.65207083,0.00989125){\makebox(0,0)[lt]{\lineheight{1.25}\smash{\begin{tabular}[t]{l}slope $+1$\end{tabular}}}}%
    \put(0.21329998,0.12343753){\makebox(0,0)[lt]{\lineheight{1.25}\smash{\begin{tabular}[t]{l}$t_0$\end{tabular}}}}%
    \put(0.62569423,0.12123531){\makebox(0,0)[lt]{\lineheight{1.25}\smash{\begin{tabular}[t]{l}$t_1$\end{tabular}}}}%
    \put(0.82749062,0.16156015){\makebox(0,0)[lt]{\lineheight{1.25}\smash{\begin{tabular}[t]{l}$1$\end{tabular}}}}%
    \put(-0.00394721,0.16051277){\makebox(0,0)[lt]{\lineheight{1.25}\smash{\begin{tabular}[t]{l}$0$\end{tabular}}}}%
    \put(0.92479402,0.35481855){\makebox(0,0)[lt]{\lineheight{1.25}\smash{\begin{tabular}[t]{l}$t$\end{tabular}}}}%
    \put(0,0){\includegraphics[width=\unitlength,page=2]{function_g.pdf}}%
  \end{picture}%
\endgroup%

        \caption{The function $g$ \Cref{lemma:cobord_between_homotopic_str_transverse_boundary} is obtained by smoothening the piecewise linear function in the picture,
        where $t_0$ and $t_1$ are the zeroes of $\cos\varphi$.
        More precisely, the smoothening is done in such a way that $\dot{g}$ is strictly positive before $t_0$ and after $t_1$, and strictly negative between them, and so that all the derivatives of $g$ are zero at $t_0$ and $t_1$.
        }
        \label{fig:function_g}
\end{figure}

An immediate consequence is the following more precise version of \Cref{cor:ConfFoliatedFilling}, which tells that the product foliation $\S^1 \times (\S^{2n+1},\xi_{ot})$ can be filled by a conformal symplectic foliation on the solid torus:

\begin{corollary}
\label{cor:CobordismSphereFilling}
Let $(\D^{2n},\Omega):\emptyset \to (\S^{2n-1},\alpha_{ot})$, $n >2$, be an almost symplectic cobordism, with $\alpha_{ot}$ an overtwisted contact form in the almost contact class of $\alpha_{st}$. Then, there exists a foliated \absholonomy exact conformal symplectic cobordism
\[(W,\F,\eta,\d_{\eta}\lambda):\emptyset \to \S^1 \times (\S^{2n-1},\alpha_{ot}).\]
Moreover, the following hold:
\begin{enumerate}[(i)]
    \item $W$ is diffeomorphic to $\S^1 \times \D^{2n}$;
    \item $(\F,\eta,\d_\eta,\omega)$ is homotopic (through almost contact structures), relative to the boundary, to the almost symplectic foliation
    \[ \big(\bigcup_{\theta \in \S^1} \{ \theta \} \times \D^{2n}, \,\Omega\big);\]
    \item On any collar neighborhood $(\varepsilon,0]\times \partial W$ we can arrange $\eta = 0$, and $\lambda = e^t\alpha_{ot}$.
    \item The foliation $\F$ contains two closed leaves, and $(\eta,\d_\eta \lambda)$ is holonomy-like away from one of the two.
\end{enumerate}
\end{corollary}

\begin{proof}
Up to homotopy of $\Omega$, one can assume that, for $\DD^{2n}_\epsilon$ a disk of very small radius $\epsilon$ centered at the origin, $(M\coloneqq \SSS^{1}\times\DD^{2n},\F=\ker\d\theta,\Omega)$ restricts over $T\coloneqq \SSS^1\times\DD^{2n}_\epsilon$ to $(\F=\ker\d\theta,\omega_{std})$, with $\omega_{std}$ the standard symplectic structure on $\RR^{2n}$.
In particular, $\Omega$ gives a homotopy of contact forms from $\alpha_{std}$ to $\alpha_{ot}$ on $M\setminus \mathring{T}\simeq \SSS^1\times[0,1]\times\SSS^{2n-1}$.
One can then simply apply \Cref{lemma:cobord_between_homotopic_str_transverse_boundary} to get a homotopy, relative boundary, from $\Omega$ to a $\vert$holonomy$\vert$-like exact conformal symplectic structure on $M\setminus T$.
This then glues well to $T= (\SSS^1\times\DD^{2n}_\epsilon,\F=\ker\d\theta,\omega_{std})$, which is holonomy-like exact conformal symplectic, along their common boundary, thus concluding the proof.
\end{proof}

Although we do not need it for our applications, let us point out an explicit result confirming what we claimed above, namely that the fact that the concave boundary component of almost symplectic \emph{foliated} cobordisms is overtwisted is not an essential assumption in order to obtain leafwise conformal symplectic structures on them. 
More precisely, we now prove that \Cref{lemma:cobord_between_homotopic_str_transverse_boundary} remains true even if both the positive and negative boundary are tight:

\begin{corollary}
\label{cor:lemma_true_even_both_tight}
Let $\xi_0=\ker\alpha_0$ and $\xi_1=\ker\alpha_1$ be contact structures on $M^{2n}$, with $n>2$, homotopic as almost contact structures via $(\alpha_t,\omega_t)$.
Then, there exists an exact $\vert$holonomy$\vert$-like conformal symplectic foliated cobordism
\[ \left(\S^1 \times [0,1] \times M, \F,\eta,\d_\eta\lambda\right),\]
from $\S^1\times (M,\xi_0)$ to $\S^1 \times (M,\xi_1)$.
Furthermore:
\begin{enumerate}[(i)]
\item  $(\F,\eta,\d_\eta\lambda)$ is in the same almost contact class (relative boundary) as $ \left(\ker \d \theta, \d t \wedge \alpha_t + \omega_t\right)$, where $t$ denotes the interval coordinate on $\S^1 \times [0,1] \times M$.
\item Near the boundary, $\lambda$ equals $e^t\alpha_0$, resp.\ $e^t\alpha_1$, and $\eta=0$.
\item $(\F,\eta,\d\lambda)$ is holonomy-like away from a closed leaf.
\end{enumerate}
\end{corollary}

\begin{proof}
Consider $(V\coloneqq \SSS^1\times[0,1]_t\times M,\G=\ker\d\theta,\Omega= \d t \wedge \alpha_t + \omega_t)$.
Let also $N=\SSS^1_\theta\times\DD^{2n}$ be a neighborhood of a (positively) transverse curve $\gamma$ where $\G=\ker\d\theta$ and $\Omega = \omega_{std}$, with $\omega_{std}$ the standard symplectic structure on $\RR^{2n}$.
Up to homotopy (relative boundary of $V$) of $\Omega$, one can arrange that $\Omega$ restricts to $\d r \wedge \alpha_{ot} + \d\alpha_{ot}$ on a neighborhood $\SSS^1\times(-\epsilon,\epsilon)_r\times\SSS^{2n-1}$ of $\partial N = \SSS^1\times\SSS^{2n-1}$ inside $V$ (here, $\partial_r$ points outwards from $N$), with $\alpha_{ot}$ defining an overtwisted contact structure $\xi_{ot}$ on $\SSS^{2n-1}$ which is in the same almost contact class as the standard tight contact structure.

Then, \Cref{cor:CobordismSphereFilling} tells that $(\G,\Omega)$ on  $N$ can be homotoped, among almost contact structures and relative boundary (of $N$), to a $\vert$holonomy$\vert$-like exact conformal symplectic foliation on $N$, which is moreover holonomy-like away from a closed leaf.
Moreover, on the complement $V\setminus N = \SSS^1\times ([0,1]\times M \setminus \DD^{2n})$ there is a leafwise homotopy (relative boundary) of almost symplectic structures ending at an exact holonomy-like conformal symplectic leafwise structure according to  \Cref{thm:MakingCobordismSymplectic} (c.f.\ \Cref{rmk:hol_like_EM}).
Composing these two homotopies then give the desired homotopy on $V$.
\end{proof}

We now state another corollary of the previous results,
which is one of the main lemmas for the proof of \Cref{prop:conf_sympl_fol_taut} in \Cref{sec:proof_main_thm}.

\begin{lemma}
\label{lemma:alm_sympl_cap_to_foliated_conf_sympl_cap}
Let $(\alpha,\omega)$ be an almost contact structure on $M^{2n+1}$, with $n> 2$. 
Let $\xi_\pm := \ker \alpha_\pm$ be overtwisted contact structures in the almost contact class $(\pm \alpha,\omega)$.
Then, there is an exact $\vert$holonomy$\vert$-like conformal symplectic foliated cobordism 
\[
 (W,\F,\eta,\d_\eta\lambda): \S^1 \times (M,\xi_+) \sqcup \S^1 \times (\overline{M},\xi_-) \to \emptyset.
\]
Moreover:
\begin{enumerate}[(i)]
    \item The cobordism $W$ is diffeomorphic to $\S^1 \times [0,1] \times M$;
    
    \item Let $(\alpha_t,\omega_t)$ be any almost contact homotopy from $(\alpha_+,\d \alpha_+)$ to $(-\alpha_-,\d \alpha_-)$.
    Then the conformal symplectic foliation $(\F,\eta,\omega)$ is homotopic, relative to $\partial W$ and among almost contact structures, to the almost symplectic foliation
    \[ \big( \bigcup_{\theta \in \S^1} \{\theta\} \times [0,1] \times M, \d t \wedge \alpha_t + \omega_t\big).\]
    
    \item Near the left/right boundary, $\lambda$ coincides with $\d(e^{\pm}\alpha_{\pm})$ respectively, and $\eta=0$.
    
    \item $(\F,\eta,\d_\eta\lambda)$ is holonomy-like away from two closed leaves.
    
\end{enumerate}
\end{lemma}

\begin{figure}[t]\centering
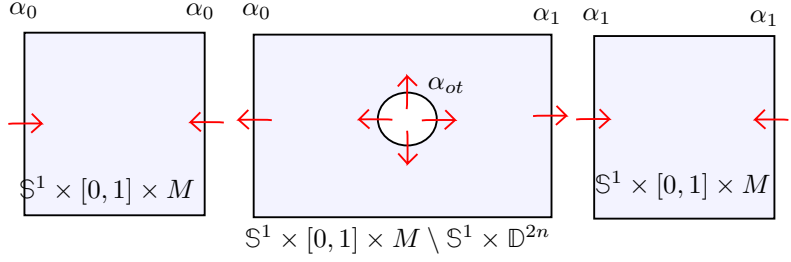
\caption{The picture depicts the several cobordisms which are stacked together in order to prove \Cref{lemma:alm_sympl_cap_to_foliated_conf_sympl_cap}.
(The foliated cobordism $\emptyset\sim (\S^{2n-1},\xi_{ot})\times \S^1$ used to fill the remaining boundary component is not depicted for simplicity.)
}
\label{fig:pieces_lemma}
\end{figure}

\begin{proof}
The idea of the proof is completely analogous to that of \Cref{lemma:cobord_between_homotopic_str_transverse_boundary}, with an additional minor ``complication''.
First, one inverts the direction of the Liouville vector fields at both boundaries by conformal symplectic turbulization, as done in the proof of \Cref{lemma:cobord_between_homotopic_str_transverse_boundary}.
However, this results here in the product of $\SSS^1$ with an almost symplectic filling of $(M,\xi_0)\sqcup(\overline{M},\xi_1)$, which has to be transformed to a foliated conformal symplectic filling, via a homotopy relative to the boundary.
The idea is then to apply \Cref{thm:MakingCobordismSymplectic} (c.f.\ \Cref{rmk:hol_like_EM}) after removing a neighborhood of a transverse curve, then \Cref{cor:CobordismSphereFilling} to fill such neighborhood again.
The building blocks are described in \Cref{fig:pieces_lemma}.
For completeness, we also give the details of this proof using global formulas using what done in the proof of \Cref{lemma:cobord_between_homotopic_str_transverse_boundary}.
\\

As done with the first two homotopies in the proof of \Cref{lemma:cobord_between_homotopic_str_transverse_boundary}, one can first homotope rel.\ boundary (among almost contact structures) the almost symplectic foliation $(\ker\d\theta,\Omega)$ to an almost symplectic foliation $(\calF,\Omega')$ such that:
\begin{itemize}
    \item $\calF$ is the desired (i.e.\ as in the conclusion of the statement) smooth foliation, obtained by turbulizing $\ker\d\theta$ near both the boundary components;
    \item for some $\delta>0$, $\Omega'$ is a $\vert$holonomy$\vert$-like exact conformal symplectic leafwise structure $(\eta,\d_\eta\lambda)$ on $\SSS^1\times[0,\delta]\times M \cup \SSS^1\times[1-\delta,1]\times M$;
    \item $(\eta,\d_\eta\lambda)$ is holonomy-like away from a closed leaf;
    \item near $\SSS^1\times\{\delta\}\times M$, $\eta=0$ and $\lambda = e^{-t}\alpha_0$;
    \item  near $\SSS^1\times\{1-\delta\}\times M$, $\eta=0$ and $\lambda = e^{t}\alpha_1$;
    \item $\calF=\ker(-\d\theta)$ on $\SSS^1\times[\delta,1-\delta]\times M$;
    \item $\Omega'$ is $\theta-$invariant on $\SSS^1\times[\delta,1-\delta]\times M$.
\end{itemize}

Denote $X\coloneqq \SSS^1\times[\delta,1-\delta]\times M$ for simplicity.
Consider now a neighborhood of a transverse curve of the form $N\coloneqq \SSS^1_\theta\times \DD^{2n}$, where $\F=\ker\d\theta$.
We also assume, up to a $\theta-$invariant homotopy, that $\Omega'$ looks like a piece of positive half-symplectization of $(\SSS^{2n+1},\xi_{ot})$ near the boundary of $N$.
We then remove the interior of $N$ from $X$, and denote $X'$ the resulting smooth manifold, as well as $(\calF',\Omega'')$ the resulting almost symplectic foliation.
Then, by (a $\theta-$invariant version of) \Cref{thm:MakingCobordismSymplectic} (c.f.\ \Cref{rmk:hol_like_EM}), $\Omega''$ is homotopic, rel.\  the three boundary components of $X'$, to an holonomy-like exact conformal symplectic foliated structure on $(X',\calF')$.

Lastly, one can use \Cref{cor:CobordismSphereFilling} in order to homotope (among almost contact structures, and relative boundary) in the remaining piece, i.e.\ the almost symplectic foliation in the $\SSS^1\times\DD^{2n}$, to an $\vert$holonomy$\vert$-like exact conformal symplectic foliation, which is holonomy-like away from a closed leaf.
This concludes the proof of \Cref{lemma:alm_sympl_cap_to_foliated_conf_sympl_cap}.
\end{proof}

\subsection{Proof of \Cref{thm:exist_conf_sympl_fol}}
\label{sec:proof_h_principle}

If $\dim M = 3$ the result follows from \cite{Thurston76}, so we assume $\dim M \geq 7$. 
Let $(\xi := \ker \alpha, \omega)$ be an almost contact structure. 
By \cite{Mei17} there exists a homotopy $\xi_t$, $t \in [0,1]$ from $\xi$ to a minimal foliation $\G$ (i.e. all leaves dense). 
We now claim that $\G$ admits a (leafwise) almost symplectic structure $\Omega$.
Indeed, we can interpret $\xi_t$ as a vector bundle $\widehat{\xi}$ on $M \times [0,1]$, and using parallel transport it follows that
\[\widehat{\xi} \simeq \pi^*\xi_0,\]
where $\pi:M \times [0,1] \to M$ is the projection onto the first factor. Thus $\omega$ induces an almost symplectic structure on $\widehat{\xi}$, and in particular on $\G = \xi_1$.
Since $\G$ is minimal,
it is taut, hence applying \Cref{prop:conf_sympl_fol_taut} to $(\G,\Omega)$ concludes the proof.

\subsection{Proof of \Cref{prop:conf_sympl_fol_taut}}
\label{sec:proof_main_thm}

The strategy of the proof is to decompose the manifold into several pieces. We construct the desired homotopy on each of the pieces, and then show they can be glued together.

Since $\G$ is taut there exist closed transverse loops intersecting every leaf. Fix two (disjoint) such loops $\gamma_\pm$, positively transverse to $\G$.
After a homotopy of $\omega$ we may assume that a small neighborhood of each curve is isomorphic to
\[ \S^1 \times (\D^{2n},\omega_{st}).\]
Let $T_\pm$ denote thickened tori around $\gamma_\pm$ contained in the neighborhood above. We choose (orientation preserving) coordinates
\begin{equation}\label{eq:MainCylinderCoordinates}
    T_- \simeq \S^1 \times [0,1] \times \S^{2n-1},\quad \text{and,}\quad T_+ \simeq \S^1 \times [0,1] \times \overline{\S^{2n-1}}.
\end{equation}
Note that in these coordinates $\partial_t$ (corresponding to the interval coordinate) points away from $\gamma_-$ and towards $\gamma_+$.

Let $\alpha_{ot,\pm} \in \Omega^1(\S^{2n-1})$ be overtwisted contact forms in the almost contact class $(\mp \alpha_{st},\d \alpha_{st})$. In particular, $\alpha_+$ is a positive contact form on $\overline{S^{2n-1}}$. Then, after another homotopy, we can arrange that
\[ \Omega|_{T_\pm} = \d t \wedge \alpha_{ot,\pm} + \d \alpha_{ot,\pm}.\]
The complement of the two cylinders consists of three connected components. These can be interpreted as foliated almost symplectic cobordisms between contact foliations:
\begin{equation}
\label{eqn:cobordisms}
\emptyset \to \S^1 \times (\S^{2n-1},\alpha_{ot,+}),\quad \S^1 \times (\S^{2n-1},\alpha_{ot,+}) \to \S^1 \times (\overline{\S^{2n-1}},\alpha_{ot,-}),\quad \text{and}\quad\S^1 \times (\overline{\S^{2n-1}}, \alpha_{ot,-}) \to \emptyset.
\end{equation}
On the first cobordism we use \Cref{cor:CobordismSphereFilling} to homotope $(\G,\omega)$, relative to the boundary, to an exact conformal symplectic foliation. 
Secondly, according to \Cref{cor:BMRelative} the middle cobordism can be homotoped to an exact conformal symplectic foliated cobordism. 
Note that this homotopy is not relative to the boundary; instead, it induces a homotopy of overtwisted almost contact foliations on the boundary ending at an honest (overtwisted) contact foliation.
Lastly, in the third cobordism we can homotope $(\G,\omega)$, relative to the boundary, to an exact conformal symplectic foliation using \Cref{lemma:alm_sympl_cap_to_foliated_conf_sympl_cap} and \Cref{cor:CobordismSphereFilling}.

It remains to extend the homotopies over $T_\pm$. Since the arguments are the same on both cylinders we focus on $T_-$. In the coordinates of \Cref{eq:MainCylinderCoordinates} any homotopy of $\Omega$ can be written as
\[ \Omega_s = \d t \wedge \alpha_{\theta,t,s} + \omega_{\theta,t,s},\]
for a $3$-parameter family of almost contact structures $(\alpha_{\theta,t,s},\omega_{\theta,t,s})$, where $(\theta,t,s) \in \S^1 \times I^2$, on $\S^{2n-1}$. 

Recall that an admissible form (\Cref{eq:admissibleform})  is equivalent to a transverse vector field. Hence the homotopy from \Cref{cor:BMRelative} determines a homotopy of vector fields. We want to assume that $\partial_t$ smoothly extends these vector fields.
Although this need not be true a priori, we can always arrange it by reparametrizing $T_-$. 
Consequently we may assume that the family of contact forms $(\alpha_{\theta,t,s},\omega_{\theta,t,s})$ for
\[ (\theta,t,s) \in \S^1 \times (I \times \{0\} \cup \{0,1\} \times I),\]
admits an overtwisted basis (on each of the smooth pieces of the parameter space, coinciding at the corners).
Using the h-principle from \cite[Theorem 1.6]{BEM15} we can extend it to a family of contact forms for $(\theta,t,s) \in \S^1 \times I^2$ as desired. 
Finally, we apply \Cref{lem:ConformalRelativeTrick} below to $\Omega_1$. This gives a homotopy, relative to the boundary of $T_-$ to a leafwise exact conformal symplectic structure, thus concluding the proof.
\hfill \qedsymbol

\begin{lemma}\label{lem:ConformalRelativeTrick}
Consider a family of contact forms $\alpha_{t,k}$, $(t,k)\in [0,1] \times K$ on $M$.
Then, the $K$-family of almost symplectic structures
\[\Omega_k = d t \wedge \alpha_{t,k} + \d\alpha_{t,k},\quad  k \in K\]
on $[0,1] \times M$ is homotopic relative to the boundary to a $K$-family of exact conformal symplectic structures.
Moreover, if $\alpha_{t,k}$ is independent of $t$, for $t$ near the endpoints, the homotopy is relative to a neighborhood of the boundary.
\end{lemma}

\begin{figure}[thp]
    \centering
    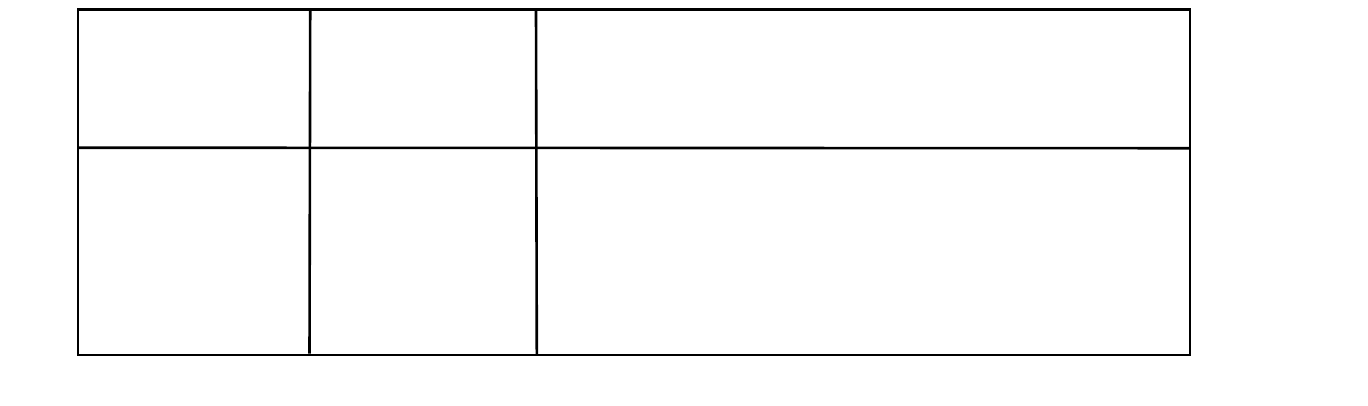
    \caption{The picture describes the regions in the manifold and which result is used to obtain the homotopy. The middle unlabeled piece is the part where the $h-$principle from \cite{BerMei} is applied.
    }
\end{figure}

\begin{proof}
By a homotopy relative to the endpoints, we can arrange that the family $\alpha_t$ is constant for $t$ near the endpoints.
Next observe that
\[ \d_{-C \d t}\alpha_t = \d t \wedge (C \alpha_t + s\dot{\alpha}_t) + \d \alpha_t, \]
is non-degenerate for any $s\in[0,1]$ provided that $C>0$ is sufficiently large, which we hence assume to be the case. 
Let $f \in C^\infty([0,1]_t)$, with $f\leq -1$, be equal to $-C$ on $\supp \dot{\alpha}_t$ and $-1$ near the boundary. 
Then $\d_{\d f}\alpha_t$ is conformal symplectic, and the linear homotopy to $\d t \wedge \alpha_t + \d \alpha_t$ is relative to the boundary.
\end{proof}

\subsection{Deformations to contact structures}
\label{sec:deformation_to_contact}

A natural question is that of whether the foliations from \Cref{thm:exist_conf_sympl_fol} can be deformed into contact structures in the spirit of \cite{EliThu98}. 
In higher dimensions there are several notions of convergence/deformation which are interesting to consider, but possibly the simplest is the following. It is the analogue of linear deformation in dimension three as defined in \cite{EliThu98}.

\begin{definition}[{\cite[Definition 2.2.5]{ToussaintThesis}}]\label{def:typeIdeformation}
A foliation $\F = \ker \alpha$ on $M^{2n+1}$ is said to admit a \emph{Type I} deformation if there exists a $1$-parameter family $\alpha_t \in \Omega^1(M)$ such that $\alpha_0 = \alpha$ and
\[ \alpha_t \wedge \d \alpha_t^n = t^n f_t \vol_M,\]
for $f_t \in C^\infty(M)$ satisfying $f_0 > 0$ and $\vol_M$ a (positive) volume form.
\end{definition}
The name refers to the fact that a deformation $\alpha_t$ is Type I if and only if its first order approximation is. 
More precisely, given $\alpha_t$ we can define its linearization:
\[ \alpha_t^{lin} := \alpha_0 + t \frac{\d}{\d t}\Big|_{t=0} \alpha_t,\]
and a simple computation shows $\alpha_t$ is Type I if and only if its linearization is.
Note however, that linear contact deformations (i.e. $\alpha_t := \alpha + t \beta$ for some $\alpha$ and $\beta$) are not necessarily Type I.

\begin{theorem}[{\cite[Lemma 6.1]{BerMei}, \cite[Theorem 2.2.13]{ToussaintThesis}}]\label{thm:TypeIDeformations}
Consider a foliation $\F$ with associated holonomy form $\mu_\F \in \Omega^1(\F)$. Then $\F$ admits a Type $I$ deformation if and only if there exists $\lambda \in \Omega^1(\F)$ such that $\d_{\mu_\F} \lambda$ is a leafwise conformal symplectic structure.
\end{theorem}

This theorem motivates the following definition:
\begin{definition}\label{def:TypeIConformalSymplectic}
A conformal symplectic foliation $(\F,\eta,\omega)$ admits a Type I deformation if $\omega  = \d_{\mu_\F} \lambda$ for some $\lambda \in \Omega^1(\F)$.
\end{definition}

It is clear that the conformal symplectic foliations constructed in \Cref{thm:exist_conf_sympl_fol} do not admit a Type I deformation. Indeed, these foliations always contain leaves on which $\eta = - \mu_\F$. 
In fact, there is a deeper obstruction than just the fact that $\eta$ and $\mu_\F$ have opposite signs to the existence of Type I deformations, as we now describe.

Recall first that given a contact manifold $(M,\alpha)$, \Cref{thm:ConformalTurbulization} allows us to construct a conformal symplectic foliation
\begin{equation}\label{eq:concaveconcavemodel}
    \left(\S^1 \times [-1,1] \times M, \F, \eta, \omega\right)
\end{equation}
satisfying the following conditions:
\begin{enumerate}[(i)]
    \item The foliation $\F$:
    \begin{enumerate}
        \item has a single closed leaf diffeomorphic to $\S^1 \times M$, and $\mu_\F = \pm \d\theta$ restricted to this leaf;
        \item equals $\ker \d \theta$ near the boundary (where $\theta \in \S^1$);
        \item every leaf accumulates onto the closed leaf.
\end{enumerate}
    \item The leafwise conformal symplectic structure has convex contact boundary (\Cref{def:ConformalContactType}).
\end{enumerate}
We will call any conformal symplectic foliation as above a \emph{concave-concave turbulization model} for $(M,\alpha)$.

Observe that near the boundary (where $\F = \ker \d \theta$) we have $\mu_\F = 0$, and on the interior $(\eta,\omega)$ is not prescribed by the model. 
Hence, a priori by using a different model than in the proof of \Cref{thm:exist_conf_sympl_fol} we can avoid the problem that $\eta = -\mu_\F$.
However, as shown in the following lemma,  the existence of a model which admits a Type I deformation is often still obstructed. Recall, that a closed contact manifold is called \emph{co-fillable} if it is a connected component of a symplectic manifold with (possibly disconnected) convex contact boundary.

\begin{lemma}
\label{lem:obstruction_deformation}
Let $(M,\alpha)$ be a closed contact manifold which is not co-fillable. Then, no concave-concave turbulization model for $(M,\alpha)$ admits a Type I deformation.
\end{lemma}
In particular, the lemma applies to the proof of \Cref{thm:exist_conf_sympl_fol}, where we use a concave-concave model for $(\S^{2n-1},\alpha_{ot})$. 
Indeed, by \cite{BEM15,Nie06} these contact manifolds are not co-fillable.
\begin{proof}
We prove the case that $\mu_\F = \d\theta$ along the closed leaf. The proof of the other case is similar.
Assume by contradiction that there exists a type I deformation, so that $\omega = d_{\mu_\F} \lambda$ for some $\lambda \in \Omega^1(\F)$. In particular, on the closed leaf $\S^1 \times M$ we have $\omega = \d \lambda - \d \theta \wedge \lambda$ so that $\lambda$ defines a negative contact form on $M$.

Now consider a non-compact leaf $(-\infty,0] \times M$ of $\F$, intersecting the right boundary. 
This leaf accumulates onto the compact leaf. 
Hence, by continuity $\lambda|_{\{t_0\} \times M}$ is a negative contact form for some $t_0 \ll 0$. 
On the other hand, by definition of the model, $\lambda = e^t\alpha$ on $(-\varepsilon,0] \times M$. In particular $\lambda|_{\{0\} \times M}$ is a positive contact form. 

Thus, $(M,\alpha)$ is a connected component of the conformal symplectic manifold $([t_0,0] \times M, \d_{\mu_\F} \lambda)$. 
Since, $\mu_\F$ is exact away from the compact leaf, this implies $(M,\alpha)$ is co-fillable giving a contradiction.
\end{proof}

\subsubsection{Deformation to contact structures for non taut foliations}
\label{sec:deformation_to_contact_non_taut}

Consider a taut almost symplectic foliation on a manifold $M^{2n+1}$, $n \geq 2$. On the complement of two curves \cite[Theorem B]{BerMei} yields a taut conformal symplectic foliation that admits a Type I deformation. As already expected in \cite[Remark 6.4]{BerMei}, tautness is not necessary if we remove an additional curve:
\begin{proposition}
\label{prop:existence_deformations_contact_away_tubes}
Let $(M^{2n+1},\zeta,\mu)$, $n \geq 2$, be an almost contact manifold. Fix any three curves, two of which are parallel. Then, on the complement of these curves, $(\zeta,\mu)$ is homotopic to a non-taut, exact conformal symplectic foliation which admits a Type I deformation. 
\end{proposition}

\begin{proof}
The proof essentially follows from that of \Cref{prop:conf_sympl_fol_taut}. 
(We point out that, even though \Cref{prop:conf_sympl_fol_taut} assumes $\dim M \geq 3$, this is only used when applying \Cref{cor:CobordismSphereFilling}.
As we now describe, these tubes will be removed from the ambient manifold so the proof goes through in dimension $5$.)

The start of the proof of \Cref{prop:conf_sympl_fol_taut} should be modified as follows.
We first choose two curves $\gamma_\pm$ and homotope the almost contact structure to be a symplectic foliation near them, and positively transverse to them. 
Then, we apply \cite{Mei17} relatively to these foliated neighborhoods (c.f.\ \cite[Remark 6.4]{BerMei}), in such a way to deform the almost contact structure to a minimal almost symplectic foliation relative to the foliated neighborhoods of $\gamma_\pm$.

All the results (namely \Cref{cor:CobordismSphereFilling}, \Cref{cor:BMRelative}, and \Cref{lemma:alm_sympl_cap_to_foliated_conf_sympl_cap}) used in the proof of \Cref{prop:conf_sympl_fol_taut} give $\vert$holonomy$\vert$-like exact conformal symplectic foliations. Moreover they are holonomy-like away from closed leafs coming from the turbulization construction.
The only part where this $\vert$holonomy$\vert$-like property has not been explicitly stated is in the construction of the homotopy in the regions $T_\pm$.
However, there the holonomy-like property follows from the same consideration as in \Cref{rmk:hol_like_EM}.

In total the argument produces three such closed leaves, all resulting from \Cref{cor:CobordismSphereFilling}.
As such, it is enough to consider the three tubular neighborhoods of transverse curves to which \Cref{cor:CobordismSphereFilling} is applied (two of which are parallel to $\gamma_+$) in the proof of \Cref{prop:conf_sympl_fol_taut}.
\end{proof}

Instead of removing curves, we can try to deform the conformal symplectic foliations from \Cref{thm:exist_conf_sympl_fol} on the whole manifold. As we show now, this produces a Type I linear deformation to ``singular contact structures''. That is, a hyperplane field which is a (positive) contact structure on the complement of an (not necessarily connected) embedded hypersurface. 

\begin{proposition}
\label{prop:deformation_to_singular_contact}
The conformal symplectic foliations constructed in \Cref{thm:exist_conf_sympl_fol} admit a Type I deformation to singular contact structures, with singularities along three closed leafs. 
\end{proposition}

\begin{proof}
The Type I linear contact deformation has been already described in the proof of \Cref{prop:existence_deformations_contact_away_tubes} away from 
solid tori in which \Cref{cor:CobordismSphereFilling} is applied.
In fact, the same argument proves the existence of a Type I deformation everywhere except on 
neighborhoods of three closed leaves in these tori, which correspond to turbulizations where both sides are concave contact boundaries.
Hence it suffices to describe how to extend the deformation near such closed leaves.

Consider the manifold $\S^1 \times[-1,1]\times \SSS^{2n-1}$ with coordinates $(\theta,t,x)$, and $\alpha \in \Omega^1(\SSS^{2n-1})$ a positive contact form.
We know that on the boundary of this manifold, the previously constructed exact holonomy-like conformal symplectic structures satisfies $\eta=0$ and $\lambda = e^{-t}\alpha$.
The turbulized smooth foliation obtained as in \Cref{thm:exist_conf_sympl_fol} can be described in this region as:
\[ 
\F \coloneqq \ker \gamma,
\quad 
\gamma \coloneqq f(t) \d \theta - g(t) \d t,
\]
where $f,g:\R \to \R$ satisfy the following properties:
\begin{enumerate}[(i)]
    \item $f$ has a single zero at $t =0$ with $\dot{f}(0)=1$, and $f = \begin{cases} -1 & \text{$t$ near $-1$} \\ 1 & \text{$t$ near $1$}\end{cases}$.
    \item $g \geq 0$, supported inside $(-1/2,1/2)$, and $g(0) = 1$.
\end{enumerate}
In order to guarantee the last property, an explicit possible formula for $k$ near $t=0$, where $f$ can be assumed wlog to be $f(t)=t$, is for given by $k(t)=1-\sqrt{1-t^2}$. 
It is then easy to extend such local choice away from this neighborhood of $t=0$ to a function satisfying the above conditions.

Furthermore, choose a function $k\colon[-1,1]\to \RR$ as in \Cref{fig:graph_f_k}, satisfying: 
\begin{itemize}
    \item $k=0$ at $t=0$ and $k>0$ otherwise,
    \item $k=e^{-t+1}$ near $t=-1$ and $k=e^{t-1}$ near $t=1$,
    \item $\dot{k} f - \dot{f}k\geq0$, with $=0$ only for $t=0$.
\end{itemize}
We then consider the linear deformation $\gamma_s\coloneqq \gamma + sk\alpha$.
To see it defines a Type I deformation observe that:
\begin{align*}
\gamma_s \wedge \d\gamma_s^n
& =
n \dot{f} s^{n} k^{n} \d t\wedge \d \theta \wedge  \alpha \wedge \d \alpha^{n-1}
+
n f s^{n} \dot{k} k^{n-1} \d\theta \wedge \d t \wedge \alpha \wedge \d\alpha^{n-1}
\\
&=
n s^{n}k^{n-1} (\dot{k} f - \dot{f}k) \d \theta \wedge \d t\wedge \alpha \wedge \d\alpha^{n-1}.
\end{align*}

By the previous choice of $k$, the local linear deformation $\gamma_s$ is then a contact deformation away from $t=0$, i.e.\ away from the closed leaf coming from this turbulization, and it glues well (up to rescaling) to the deformation previously described away from this turbulization region.
\end{proof}

\begin{figure}[th]
    \centering
\begingroup%
  \makeatletter%
  \providecommand\color[2][]{%
    \errmessage{(Inkscape) Color is used for the text in Inkscape, but the package 'color.sty' is not loaded}%
    \renewcommand\color[2][]{}%
  }%
  \providecommand\transparent[1]{%
    \errmessage{(Inkscape) Transparency is used (non-zero) for the text in Inkscape, but the package 'transparent.sty' is not loaded}%
    \renewcommand\transparent[1]{}%
  }%
  \providecommand\rotatebox[2]{#2}%
  \newcommand*\fsize{\dimexpr\f@size pt\relax}%
  \newcommand*\lineheight[1]{\fontsize{\fsize}{#1\fsize}\selectfont}%
  \ifx\svgwidth\undefined%
    \setlength{\unitlength}{209.27125067bp}%
    \ifx\svgscale\undefined%
      \relax%
    \else%
      \setlength{\unitlength}{\unitlength * \real{\svgscale}}%
    \fi%
  \else%
    \setlength{\unitlength}{\svgwidth}%
  \fi%
  \global\let\svgwidth\undefined%
  \global\let\svgscale\undefined%
  \makeatother%
  \begin{picture}(1,0.70196691)%
    \lineheight{1}%
    \setlength\tabcolsep{0pt}%
    \put(0,0){\includegraphics[width=\unitlength,page=1]{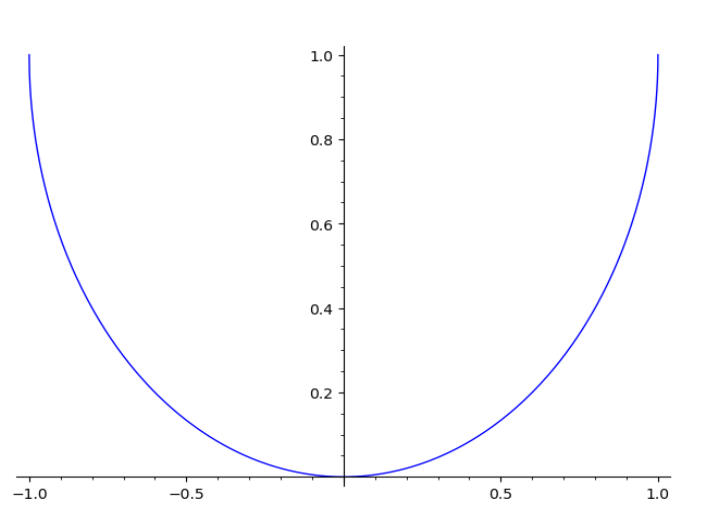}}%
    \put(0.93841492,0.03253584){\makebox(0,0)[lt]{\lineheight{1.25}\smash{\begin{tabular}[t]{l}$f$\end{tabular}}}}%
    \put(0.44629189,0.66928206){\makebox(0,0)[lt]{\lineheight{1.25}\smash{\begin{tabular}[t]{l}$k$\end{tabular}}}}%
    \put(0.14999794,0.27452857){\makebox(0,0)[lt]{\lineheight{1.25}\smash{\begin{tabular}[t]{l}$(f,k)$\end{tabular}}}}%
    \put(0.02618038,0.18372787){\color[rgb]{1,0,0}\makebox(0,0)[lt]{\lineheight{1.25}\smash{\begin{tabular}[t]{l}$(\dot{f},\dot{k})$\end{tabular}}}}%
    \put(0,0){\includegraphics[width=\unitlength,page=2]{graph_f_k.pdf}}%
  \end{picture}%
\endgroup%

    \caption{Parametric graph of the functions $f$ and $k$}
    \label{fig:graph_f_k}
\end{figure}

\subsubsection{Proof of \Cref{thm:deformations_self_round_connected_sums}}
\label{sec:deformation_for_round_connected_sums}

By a homotopy of $(\zeta,\mu)$ we can assume that there is a neighborhood $N \simeq \S^1 \times \D^{2n}$ of the curve $\gamma$ on which
\[ \zeta = \bigcup_{\theta \in \S^1} \{ \theta \} \times \D^{2n},\quad \mu = \omega_{st}.\]
Applying \cite{Mei17} relative to this neighborhood (see \cite[Remark 6.4]{BerMei}) we can homotopy $\zeta$ to a taut foliation $\G$ (equal to the product foliation on the neighborhood above).
Note that at this point we are in the hypothesis of  \Cref{rmk:starting_from_fol}.

We fix two parallel copies $\gamma_\pm$ of the curve $\gamma$ inside $N$, and positively transverse to $\G$. 
Then, as done in the proof of \Cref{prop:conf_sympl_fol_taut}, we apply \cite{BerMei} to the complement of $\gamma_\pm$, and compose it with the foliated cobordisms $T_\pm$ described in that proof, in order to obtain a foliated conformal symplectic cobordism
\[
\S^1\times (\S^{2n-1},\alpha_{ot,-})
\to \S^1\times (\overline{\S^{2n-1}},\alpha_{ot,+}).
\]
Moreover, we can arrange the conformal symplectic foliation to be leafwise exact and holonomy-like.

Consider now an orientation reversing diffeomorphism $\psi:\S^{2n-1} \to \S^{2n-1}$. This induces an orientation \emph{preserving} diffeomorphism
\[\Psi:\S^1 \times \S^{2n-1} \tois \S^1 \times \overline{\S^{2n-1}},\quad (\theta,x) \mapsto (\theta,\psi(x)).\]
As such, the above cobordism can also be interpreted as a foliated conformal symplectic cobordism
\begin{equation}\label{eq:thm18cobordism}
\S^1\times (\S^{2n-1},\alpha_{ot,-})
\to \S^1\times (\S^{2n-1},\alpha'_{ot,+}=\psi^*\alpha_{ot,+}).
\end{equation}
Since $\Psi$ is orientation preserving, $\alpha_{ot,+}'$ is a positive, overtwisted contact form.

Recall now that the $0$-th homotopy group of the space of almost contact structures on $\S^{2n-1}$, is a group under taking connected sum. 
The identity element is moreover given by the homotopy class of the standard tight contact structure (which is the same as that of $\xi_{ot,-}$).
Because of the dimensional assumption $2n-1 = 4k+1 \geq 5$, it also follows from \cite{Har63} that this group is finite. 
Therefore, there exist an integer $N >0$ such that $\#_{i=1}^N \xi_{ot,+}'$ is in the same almost contact class as $\xi_{ot,-}$.
(Note that if $k=1$ we can choose $N=1$.)

Consider then the standard symplectic ball $(\D^{2n},\omega_{st})$, and remove $N$ smaller balls $\D^{2n}_\varepsilon$ from its interior. 
Then, \Cref{thm:EliMurphyMain} ,we obtain a (non-foliated) conformal symplectic cobordism $(\S^{2n-1},\xi_{ot,+}') \to \sqcup_N(\S^{2n-1},\xi_{ot,+}')$. 
We compose this with $N$ Weinstein $1$-handle attachments to reach as convex boundary $(\#_N \S^{2n-1} = \S^{2n-1},\#_N \xi_{ot,+}')$. 
As pointed out before $\#_N \xi_{ot,+}'$ is homotopic to $\xi_{ot,-}$. 
Thus, by adding another (smoothly trivial) cobordism we can obtain $(\S^{2n-1},\xi_{ot,-})$ as the concave boundary. 
The resulting Liouville cobordism $(X,\lambda_X)$ is depicted in \Cref{fig:cobordism_X}.
Note that as a smooth manifold $X$ is obtained from $[0,1]\times \S^{2n-1}$ by doing $N$ self connected sums.

\begin{figure}[h]
    \centering
    \def\svgwidth{0.8\textwidth}
    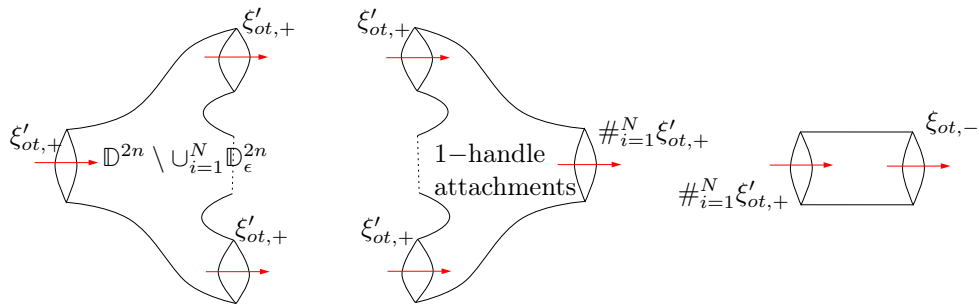
    \caption{A schematic description of the two pieces composing the cobordism $X$.}
    \label{fig:cobordism_X}
\end{figure}

According to \Cref{rmk:weak_relat_to_relat}, we can now fix any desired contact form at the boundary, at the expense of $(X,\lambda_X)$ becoming a conformal symplectic cobordism. 
Taking the product with $\S^1$ we then obtain a foliated conformal symplectic cobordism
\[ \S^1 \times (\S^{2n-1},\alpha_{ot,+}') \to \S^1 \times (\S^{2n-1},\alpha_{ot,-}').\]
Again, note that as a smooth manifold $\S^1 \times X$ is from $\S^1 \times [0,1] \times \S^{2n-1}$ by self round-connected sums.

Finally, we glue the above cobordism to the one of \Cref{eq:thm18cobordism}. This yields the desired exact, holonomy-like conformal symplectic foliation on the round connect sum of $M$ with $k$ disjoint copies of $T^2 \times \S^{4n+1}$. The existence of a Type I deformation lastly follows from \Cref{thm:TypeIDeformations}.
\hfill
\qedsymbol

\bibliographystyle{alphaurl}
\bibliography{biblio}

\end{document}